\documentclass[a4paper,dvipsnames]{scrartcl}%
\usepackage[utf8]{inputenc}
\usepackage[T1]{fontenc}
\usepackage[english]{babel}
\usepackage{amsmath}
\usepackage{amssymb,amsfonts,siunitx,commath, amsthm}
\usepackage{algorithm}
\usepackage{algpseudocode}
\usepackage{graphicx}
\usepackage{amsmath}
\usepackage{url}            % simple URL typesetting
\usepackage{booktabs}       % professional-quality tables
\usepackage{amsfonts}       % blackboard math symbols
\usepackage{nicefrac}       % compact symbols for 1/2, etc.
\usepackage{microtype}      % microtypography
\usepackage{geometry}
\usepackage{xcolor}         % colors
\usepackage{multirow}
\usepackage{makecell}
\usepackage{mathtools}
\usepackage{autonum}
\mathtoolsset{showonlyrefs}
\usepackage{bm}
\newcommand{\zb}[1]{\ensuremath{\boldsymbol{#1}}}

\newcommand{\R}{\mathbb{R}}
\newcommand{\spa}{\mathrm{span}}

\newcommand{\VV}{\mathcal{V}}
\newcommand{\dd}{\mathrm{d}}
\newcommand{\N}{\mathbb{N}}

\newcommand{\tT}{\mathrm{T}}

\newcommand{\diag}{\mathrm{diag}}

\newcommand{\uu}{\mathbf{u}}
\newcommand{\vv}{\mathbf{v}}

\DeclareMathOperator{\grad}{grad}
\DeclareMathOperator{\Retr}{Retr}
\DeclareMathOperator{\rank}{rank}

\DeclareMathOperator*{\argmin}{argmin}
\usepackage{array, booktabs}
\newcolumntype{L}[1]{>{\raggedright\let\newline\\\arraybackslash\hspace{0pt}}m{#1}}
\newcolumntype{C}[1]{>{\centering\let\newline\\\arraybackslash\hspace{0pt}}m{#1}}
\newcolumntype{R}[1]{>{\raggedleft\let\newline\\\arraybackslash\hspace{0pt}}m{#1}}

\usepackage{stackengine}
\newcommand\xrowht[2][0]{\addstackgap[.5\dimexpr#2\relax]{\vphantom{#1}}}

\newtheorem{lemma}{Lemma}[section]
\newtheorem{theorem}[lemma]{Theorem}
\newtheorem{corollary}[lemma]{Corollary}
\newtheorem{proposition}[lemma]{Proposition}
\newtheorem{remark}[lemma]{Remark}

\theoremstyle{definition}
\newtheorem{definition}[lemma]{Definition}
\newtheorem{example}[lemma]{Example}

\title{Sparse additive function decompositions facing basis transforms}

\author{
	Fatima Antarou Ba \footnotemark[1]
	\and
	Oleh Melnyk\footnotemark[1]\\%
	\and
	Christian Wald\thanks{Technische Universität Berlin, Institute of Mathematics, Straße des 17. Juni 136, D-10623 Berlin, Germany}\\%
	\and
	Gabriele Steidl\footnotemark[1]\\%
}

\begin{document}
	\maketitle
	\begin{abstract}
		High-dimensional real-world systems can often be well characterized
		by a small number of simultaneous low-complexity interactions.
		The analysis of variance (ANOVA) decomposition 
		and  the anchored decomposition 
		are typical techniques to find sparse additive decompositions  of  functions. 
		In this paper, we are interested in a setting, where these decompositions 
		are not directly spare, but become so after an appropriate basis transform.
		Noting that the sparsity of those additive function  decompositions
		is equivalent to the fact that most of its mixed partial derivatives vanish,
		we can exploit a connection to the underlying function graphs to
		determine an orthogonal transform that realizes 
		the appropriate basis change.
		This is done in three steps:  we apply singular value decomposition to minimize the
		number of vertices of the function graph, and
		joint block diagonalization techniques of families of matrices
		followed by sparse minimization  based on relaxations of the zero ''norm''
		for minimizing the number of edges.
		For the latter one, we propose and analyze minimization techniques over
		the manifold of special orthogonal matrices.
		Various numerical examples illustrate the reliability of our approach
		for functions having, after a basis transform, a sparse additive decomposition into summands with at most two
		variables.
	\end{abstract}
	
	%---------------------------------------------------------------------
	\section{Introduction}
	%---------------------------------------------------------------------
	The approximation of high-dimensional functions is a
	classical topic of mathematical analysis with rich real-world applications.
	Due to new numerical techniques arising from (stochastic) Fourier- and wavelet 
	\cite{BartelPottsSchmischke2022,  PottsSchmischke2021, lippert2022variable, Lippert2023} as well as kernel \cite{WagnerNestler2022} methods
	and deep learning approaches, the topic has recently attained increasing attention.
	For example, including information about the structure of the function class of interest
	into a deep neural network architecture can improve its approximation quality \cite{agarwal2021neural, enouen2022sparse, chang2021node}. In particular, in \cite{enouen2022sparse} block sparse additive neural network architectures are developed with sparsity patterns estimated from data. These block sparse additive neural networks show increased training speed, better memory efficiency and improved generalization.\\
	In this paper, we are interested in high-dimensional functions admitting an 
	additive decomposition into functions depending only on a few variables, i.e.,
	\begin{equation}\label{eq:add}
		f(x) = \sum_{\uu \in \mathcal S} f_\uu(x_\uu), \quad x_\uu \coloneqq (x_i)_{i \in \uu},
	\end{equation}
	where $\mathcal S$ consists of small  subsets of $\{1,\ldots,d\}$,
	and $d\gg 1$ is the dimension of the problem.
	Recently, we dealt with such decompositions in the context of multimarginal optimal transport, 
	where only special structured cost functions can be treated in an efficient way \cite{FatimaMichael2022, Beier2023}.
	
	In general, decompositions of functions of the form \eqref{add} are not unique, but there exist prominent examples in the literature having special desirable properties.
	So the analysis of variance (ANOVA) decomposition \cite{caflisch1997valuation, potts2022interpretable} and anchored \cite{hickernell2004tractability, Hickernell.2004} decompositions and their generalizations \cite{KuoSloan2010} arise in finance for option pricing, bond valuation and the pricing of collateral mortgage-backed securities problems \cite{GRIEBEL2010455}. The ANOVA decomposition is uniquely determined and the anchored decomposition up to a so-called anchored point.
	Both decompositions make it possible to analyze the different dimensions and their interactions, and to perform high dimensional integration \cite{Dick_Kuo_Sloan_2013, GRIEBEL} using quadrature methods  %\cite{Bungartz_Griebel_2004, GRIEBEL2010455, Niederreiter1992} 
	as well as infinite-dimensional integration \cite{BaldeauxGnewuch2014, GRIEBEL, KUO2017217}. Further, in \cite{FatimaJohannesGabi}, we proposed, inspired by the analysis of variance ANOVA decomposition of functions, a Gaussian-Uniform mixture model on the high-dimensional torus which relies on the assumption that the function we wish to approximate can be well explained by limited variable interactions.
	
	A further motivation for considering additive function decompositions comes from probabilistic graphical models \cite[Chapter 13]{Sullivant.2018}. According to the Hammersley-Clifford theorem, the logarithm of a continuous density function $\phi$ of a random vector can be decomposed as \eqref{eq:add} 
	for a certain class of sets $\mathcal S$ defined by conditional dependency between the components of the random vector. Finding an appropriate additive decomposition allows for inference on the interactions between the observed data, which naturally leads to applications related to causality inference and complex systems that frequently appear in computational biology \cite{Trilla-Fuertes.2023}, climate change \cite{O_Kane.2023}, speech recognition \cite{BilmesBartels2005, Bilmes.2002}, predictive modelling \cite{Kapteyn.2021}, energy systems \cite{Li.2022} and many more.
	
	One of the key features of ANOVA, anchored and Hammersley-Clifford decompositions is their minimality in the number of elements $\uu \in \mathcal S$. 
	That is, for $\uu \in \mathcal S$ it is not possible to find further decomposition $f_\uu(x_\uu) = \sum_{\vv \subsetneq \uu} f_\vv(x_\vv)$ consisting only of smaller subsets of $\uu$.  
	A composition with this minimality property is closely related to the structure of the graph associated with $f$. For twice continuously differentiable $f$, an undirected graph $G(f)$ is defined by vertices and edges  
	\begin{equation}\label{eq: graph}
		\mathcal V (f) = \{ i\in [d]: \partial_{x_i} f \neq 0\}
		\quad \text{and} \quad
		\mathcal E (f) = \{ (i,j) \in [d]^2: \partial^2_{x_i,x_j} f \neq 0 \text{ and } i \neq j\},
	\end{equation}
	respectively.  Then, $\mathcal S$ for both ANOVA and anchored decomposition is a subset of the cliques in $\mathcal C(f)$ of the graph. Furthermore, 
	$\mathcal S$ in the Hammersley-Clifford decomposition coincides with maximal cliques in $\mathcal C(f)$. 
	We will use the relation between function graphs and sparse
	additive function decompositions in Section \ref{sect:decmark}.
	
	% {\color{red} Inspired by Singular Value Decomposition (SVD) \cite{BisgardJames2021AaLA, RajwadeA2013IDUt}???}
	
	In this paper, we are interested in the setting, where the function does only admit a sparse additive decomposition after a basis transform.
	In other words, given (noisy) values of the gradient and the Hessian of $f$,
	we aim to find an orthogonal matrix $U\in O(d)$ such that the graph of the function 
	$$
	f_U \coloneqq f (U\cdot )
	$$
	has the smallest number of cliques. 
	We propose a three-step procedure. 
	First, we find an orthogonal matrix $U \in O(d)$ such that the graph of $f_U$
	has the smallest number of vertices $|\mathcal V(f_U)|$. We will see that this can be done by the singular value decomposition of the matrix having the gradient of $f$ at different values as columns.
	Then, in the second and third steps, we aim to minimize the number of edges $|\mathcal E(f_U)|$
	by joint block diagonal decomposition of the 
	Hessians of $f$ at the given points and subsequent sparsification of the 
	individual blocks.
	While we rely on results in \cite{murota2010numerical,blockdiag} for finding the joint blockdiagonalization, we apply recent sparse optimization methods
	with a relaxed $0$ ''norm'' to further sparsify the single blocks.
	The third step requires the solution of a nonconvex optimization problem over the manifold of special orthogonal matrices for which the Riemannian gradient descent \cite{boumal2023intromanifolds} or the Landing method \cite{ablin2022fast} is employed in combination with grid search.
	\\[1ex]
	\textbf{Outline of the paper.}
	In Section \ref{sect:decmark}, we consider sparse additive decompositions.
	We are interested in so-called minimal decompositions and recall that both the ANOVA and anchored decomposition are of such minimal type. We address the relation between 
	minimal additive decompositions, first and second derivatives of functions and 
	corresponding function graphs. Finally, 
	Theorem \ref{prop: term bound} gives an estimate of the influence of $\| \partial_\vv f \|_\infty$  to summands
	$f_{\uu}$ in the anchored and ANOVA decompositions for $\uu \supseteq \vv$.
	In Section \ref{sect:trafo}, we detail the three steps to obtain the desired basis transform
	towards a function which has a sparse additive decomposition.
	Section \ref{sec:MAnOpt}, deals with the optimization problem arising in Step 3,
	which requires optimization techniques over the special orthogonal group.
	Numerical results for functions admitting, after an appropriate basis transform,
	a sparse additive decomposition into summands depending at most on two variables are given in Section \ref{sect:numerics}.
	The appendix contains technical proofs, details on the algorithms
	and gives additional numerical results.
	
	%------------------------------------------------------------------------
	\section{Sparse additive decompositions and mixed derivatives}\label{sect:decmark}
	%------------------------------------------------------------------------
	In the following, let $[d] \coloneqq \{1,\ldots,d\}$. 
	We set 
	$D \coloneqq \bigotimes\limits_{i \in [d]} I_i$, where $I_i \coloneqq [a_i,b_i]$, $a_i < b_i$.
	For a subset $\uu \subseteq [d]$, we use $D_{\uu} \coloneqq \bigotimes\limits_{i \in \uu} I_i$.
	By $\lambda_\uu$, we denote the normalized Lebesgue measure on $D_{\uu}$ and 
	$\lambda \coloneqq \lambda_{[d]}$.
	Further, for $\uu =\{ i_1,\ldots, i_k\} \subseteq [d]$, 
	we use the abbreviation
	$x_\uu \coloneqq (x_{i_1},\ldots,x_{i_k})$ and
	$$
	\partial_{\uu} \coloneqq \partial_{x_{i_1}, \ldots,x_{i_k} }^{|\uu|}.
	$$
	By $\mathcal F$, we denote an appropriate subspace of function $f: D \to \mathbb R$ which will
	specified later.
	We define projection operators on $\mathcal F$ fulfilling the following assumption:
	\\
	
	\noindent
	\textbf{Assumption I:}
	The operators $P_j$, $j \in [d]$ are commuting projections on $\mathcal F$, 
	i.e.,
	$$P_i P_j f = P_j P_i f \quad \text{ for all} \quad f \in \mathcal F,$$
	$P_jf$ is independent of the $j$-th component $x_j$ 
	and $P_jf = f$ if $f$ is independent of $x_j$.
	\\
	
	We set $P_{\uu} \coloneqq \prod_{j \in \uu} P_j$.
	Based on the above projection operators,
	Kuo et al. \cite{KuoSloan2010} defined an additive decomposition of functions in $\mathcal F$
	which fulfills a certain minimization property.
	
	%-----------------------------------------------------
	\begin{theorem}\label{dec:minimal_kuo}
		Let $P_j$, $j \in [d]$ fulfill Assumption I. Then any $f \in \mathcal F$ admits a decomposition
		\begin{equation}\label{eq:general_decomposition}
			f(x)=\sum_{\uu\subseteq [d]}f_\uu(x_\uu)
		\end{equation}
		with
		\begin{equation}\label{eq: general decomposition_s}
			f_\uu:=\prod_{j\in\uu}({\text{Id}}-P_j)P_{[d]\setminus\uu}f 
			= \sum_{\vv \subseteq \uu} (-1)^{|\uu| - |\vv|} P_{[d]\setminus\uu}f.
		\end{equation}
		This decomposition is minimal in the following sense: if  for an arbitrary decomposition
		\begin{align}
			f(x) 
			&=\sum_{\uu\subseteq [d]}\tilde{f}_\uu (x_\uu),
		\end{align}
		we have
		$\tilde{f}_\vv=0$ for all $\vv \supseteq \uu$ and some $\uu \in [d]$,
		then also $f_\vv = 0$ in \eqref{eq:general_decomposition}.
	\end{theorem}
	%---------------------------------------------------
	
	To illustrate the role of the indices, we give an example.
	
	\begin{example} For $d=3$, we have 
		$$
		f = f_0 + f_1+f_2+f_3+f_{1,2} + f_{1,3}f + f_{2,3} + f_{1,2,3}
		$$
		with
		$f_0 \coloneqq f_\emptyset = P_1 P_2 P_3f = \text{const}$ and
		\begin{align}
			f_1 &= (\text{Id}  - P_1) P_2 P_3 f = P_2 P_3 f - P_1 P_2 P_3 f, 
			\\
			f_{1,2} &= (\text{Id}  - P_1)(\text{Id}  - P_2) P_3 f 
			= P_3 f - P_1 P_3 f - P_2 P_3 f + P_1 P_2 P_3 f
			\\
			f_{1,2,3} &= (\text{Id}  - P_1)(\text{Id}  - P_2)(\text{Id}  - P_3)f\\
			&= f - P_1 f - P_2 f- P_3 f + P_1 P_2 f + P_1 P_3 f + P_2 P_3 f - P_1 P_2 P_3 f
		\end{align}
		and similarly for the other summands.
	\end{example}
	
	The concrete decomposition \eqref{eq:general_decomposition} depends on the chosen projection
	which must be well defined on the space $\mathcal F$. 
	Two frequently addressed decompositions are the anchored and the ANOVA decompositions, 
	where  both projections fulfill Assumption I:
	\begin{itemize}
		\item[i)] 
		the \emph{anchored decomposition} with anchor point 
		$c = (c_1,\ldots,c_d) \in D$ 
		is determined for any $f:D \to \mathbb R$ by
		\begin{equation}\label{proj:anchored}
			P_i f = P_{i, c} f \coloneqq f(x_1, \cdots, x_{i-1}, c_i, x_{i+1}, \cdots, x_d),
			\quad i \in [d],
		\end{equation} 
		and we denote the corresponding decomposition \eqref{eq:general_decomposition} by
		\begin{align}\label{dec_c}
			f(x) 
			&=\sum_{\uu\subseteq [d]} f_{\uu,c} (x_\uu).
		\end{align}
		\item[ii)] 
		the \emph{analysis of variance} (ANOVA) \emph{decomposition}
		is given for absolutely integrable functions $f:D \to \mathbb R$ by 
		\begin{equation}\label{proj:anova}
			P_i f  = P_{i, A} \coloneqq \frac{1}{b_i - a_i}\int _{I_i} f \, \text{d} x_i,
			\quad i \in [d],
		\end{equation} 
		and we use the notation
		\begin{align}\label{dec_a}
			f(x) 
			&=\sum_{\uu\subseteq [d]} f_{\uu,A} (x_\uu).
		\end{align}
	\end{itemize}
	
	The relation between the anchor and ANOVA summands follows immediately by their definition.
	
	\begin{proposition}\label{anova:anchored}
		For $f:D \to \R$ be absolutely integrable and $\uu \subseteq [d]$, we have
		\begin{equation}
			f_{\uu,A} = \int_{D} f_{\uu, c}(x_\uu) \, \dd \lambda(c).
		\end{equation}
	\end{proposition}
	
	\begin{proof}
		By \eqref{eq: general decomposition_s}, we obtain 
		\begin{align}
			\int_{D} f_{\uu, c}(x_\uu) \dd \lambda(c) 
			&= \int_D\sum_{\vv \subseteq \uu} (-1)^{|\uu|-|\vv|} (P_{[d]\setminus \vv, c}f)(x_\vv) \, \dd \lambda(c) \\
			&= \sum_{\vv \subseteq \uu} (-1)^{|\uu|-|\vv|}\int_D (P_{[d]\setminus \vv, c}f)(x_{\vv}) \, \dd \lambda(c) \\
			&= \sum_{\vv \subseteq \uu} (-1)^{|\uu|-|\vv|}\int_{D_{[d]\setminus \vv}} f(x_\vv, c_{[d]\setminus v})\dd \lambda_{[d]\setminus \vv}(c) \\
			&= \sum_{\vv \subseteq \uu} (-1)^{|\uu|-|\vv|}  P_{[d]\setminus \vv,A} f 
			=  f_{\uu,A}(x_\uu).
		\end{align}
	\end{proof}
	
	Next, we are interested in the relation between first- and second-order derivatives of $f \in C^2(D)$
	and minimal additive decompositions \eqref{eq:general_decomposition}.
	For this, it appears to be useful to consider the undirected graphs of functions.
	
	\begin{definition}[Graph of a functions] \label{def:graph}
		To $f \in C^2(D)$, we assign the graph $\mathcal G(f)$ with vertices 
		\begin{equation}\label{eq: graph_v}
			\mathcal V (f) \coloneqq \{ i\in [d]: \partial_i f \neq 0\}
		\end{equation}
		and edges
		\begin{equation}\label{eq: graph_e}
			\mathcal E (f) \coloneqq \{ (i,j) \in [d]^2: \partial_{i,j} f \neq 0 \text{ for } i \neq j\}.
		\end{equation}
		A \emph{clique} $\mathcal C(f)$ of $\mathcal G(f)$ is a subset of vertices of $\mathcal V(f)$ such that every two distinct vertices are connected.
	\end{definition}
	
	Then we have the following relation.
	
	\begin{theorem}\label{decomposition:partial}
		Let $f\in C^2(D)$. 
		Assume that $\partial_{i,j} f=0$ for some $i \not = j$.
		Then we have in \eqref{eq:general_decomposition} that $f_\uu = 0$ 
		for all $\uu \in [d]$ with $\{i,j\}\subseteq \uu$.
		Moreover, it holds  
		\begin{equation}\label{cl}
			f(x) = \sum_{\uu\subseteq \mathcal C(f) } {f}_\uu (x_\uu) .
		\end{equation}
	\end{theorem}
	
	\begin{proof}
		We will show that $f$ has a decomposition
		\begin{equation} \label{xx}
			f =  \sum_{\uu\subseteq [d]}\tilde{f}_\uu 
			= P_{i,c} f + P_{j,c} f  - P_{\{i,j\},c} f.
		\end{equation}
		Then this decomposition fulfills $\tilde f_\uu = 0$ whenever $\{i,j\} \subseteq \uu$, and 
		Theorem \ref{dec:minimal_kuo} implies that $f_\uu = 0$. This gives the first part of the assertion.
		As a consequence, $f_\uu$ is zero unless all pairs $\{i,j\} \subseteq \uu$ admit $(i,j) \in \mathcal E(f)$, i.e., $\uu$ is a clique and we are done.
		
		It remains to prove \eqref{xx}. 
		Without loss of generality, let $\partial_{1,2} f=0$. 
		Then the Taylor expansion at $\tilde x \coloneqq (c_1,c_2,x_3, \ldots, x_d)$ reads as
		\begin{align}
			f(x) 
			&= f(\tilde x) + 
			\partial_1 f(\tilde x) (x_1 - c_1) 
			+ \partial_2 f(\tilde x) (x_2 - c_2)
			\\
			& \quad + R_{1,1}  (x_1-c_1)^2 + R_{2,2} (x_2-c_2)^2 + 2 R_{1,2} (x_1 -c_1) (x_2-c_2), 
		\end{align}
		where
		\[
		R_{i,j}  \coloneqq \int_0^1 (1-t) \partial_{i,j} f(\tilde x + t(x - \tilde x) ) \, \text{d} t, \quad i,j  \in \{1,2\} .
		\]
		By the assumption, $R_{1,2}f = 0$.
		Further, the Taylor expansion of $P_{2,c}f$ at $\tilde x$ is
		\[
		P_{2,c} f(x) = f(\bar x) =  f(\tilde x) + \partial_{1} f(\tilde x) (x_1 - c_1) + \bar R_{1,1}  (x_1- c_1)^2 
		\]
		with $\bar x \coloneqq (x_1,c_2, x_3, \ldots, x_d)$ and 
		\[
		\bar R_{1,1} 
		\coloneqq 
		\int_0^1 (1-t) \partial_{1,1} f(\tilde x + t(\bar x - \tilde x) ) \, \text{d} t,
		\] 
		and similarly for $P_{1,c} f$.
		Since $\partial_1 f$ is constant with respect to $x_2$, the same holds true for $\partial_{1,1}f$.
		Hence we conclude by their definition that $R_{1,1} = \bar R_{1,1}$.
		Consequently, we obtain
		\[
		f(x) 
		= P_{1,c} f(x) + P_{2,c} f(x) - f(\tilde x) 
		= P_{1,c} f(x) + P_{2,c} f(x)  - P_{\{1,2\},c} f(x).
		\]
		This finishes the proof.
	\end{proof}
	
	We have seen that  $\partial_{i,j} f=0$   
	implies that all summands $f_\uu$  with 
	$\uu \supseteq \{i,j\}$
	vanish in \eqref{eq:general_decomposition}.
	Next, we  want to estimate the general influence of $\| \partial_\vv f \|_\infty$  to summands
	$f_{\uu, \cdot}$ in the anchored and ANOVA decompositions for $\uu \supseteq \vv$.

	\begin{theorem} \label{prop: term bound}
		Let 
		$\vv \subseteq [d]$ 
		and 
		$f\in C^{|\vv|}(D)$.
		Then, for $\uu\supseteq \vv$, the following estimates hold true
		in the anchor and ANOVA decompositions:
		\begin{itemize}
			\item[i)] $ \|f_{\uu,c}\|_{\infty} \leq 2^{|\uu|-|\vv|} \|\partial_{x_\vv}f\|_\infty\, \lambda_\vv(D_\vv),$
			\item[ii)] $\|f_{\uu,A}\|_{\infty}\leq 2^{|\uu|-|\vv|} \|\partial_{x_\vv}f\|_\infty\, \lambda(D)\lambda_\vv(D_\vv),$
			\item[iii)] $\norm{f_{\uu,A}}_1  \leq 2^{|\uu|-|\vv|}\|\partial_{x_{\vv}}f\|_1 \, \lambda_\uu(D_\uu)\lambda_\vv(D_\vv).$
		\end{itemize}
	\end{theorem}
	
	The proof is given in Appendix \ref{app:A}.
	
	%------------------------------------------
	\section{Basis transforms towards sparse function decomposition }\label{sect:trafo}
	%------------------------------------------
	Decompositions of the form \eqref{eq:general_decomposition} 
	can be  applied for an efficient integration over high-dimensional data
	if
	\begin{itemize}
		\item[i)] the number of non-vanishing components $\uu \in [d]$, and/or
		\item[ii)] their cardinalities $|\uu|$ 
	\end{itemize}
	are small. In Theorem \ref{decomposition:partial}, we observed that these properties depend on the structure of the underlying function graph $\mathcal G(f)$ and in particular on the set of cliques $\mathcal{C}(f)$. More precisely $\mathcal G(f)$ should not contain large cliques. 
	
	In this section, we are interested in the case when the function $f$ admits only a sparse
	additive decomposition after an appropriate basis transform.
	We aim to determine such a basis transform given the first and second partial derivatives of $f$
	at various points in $D \subset \mathbb R^d$.
	Later, in our numerical experiments, we will assume that the summands
	depend at most on two of the transformed variables.
	
	Let $\mathbb B_r(d) \coloneqq \{x \in \mathbb R^d: \|x\| \le r\}$. 
	If the dimension is clear, we just write $\mathbb B_r = \mathbb B_r(d)$.
	By $M_d(\R)$, we denote the space of $d\times d$
	real-valued matrices.
	Let $O(d)$ denote the Lie group of 
	orthogonal $d \times d$ matrices and $SO(d)$ its subgroup of rotation matrices having determinant 1.
	We deal with functions $f:\mathbb B_r(d) \to \mathbb R$.
	For $U \in O(d)$, we set $f_U \coloneqq f(U \cdot)$.
	Clearly, $\mathbb B_r(d)$ is invariant under the action of orthogonal matrices, so that 
	$f_U:\mathbb B_r(d) \to \mathbb R$ is also well-defined.
	Further, we have for twice differentiable $f$ that
	\begin{equation} \label{der}
		\nabla f_U = U^\tT \nabla f(U \cdot) \quad     \text{and} \quad
		\nabla^2 f_U = U^\tT \nabla^2 f(U \cdot) U.
	\end{equation}
	
	\begin{example}
		Consider the orthogonal matrix
		$$
		U = (u_1,u_2,u_3,u_4) =
		\frac12 \left(
		\begin{array}{rrrr}
			1&1&1&1\\
			1&-1&1&-1\\
			-1&-1&1&1\\
			-1&1&1&-1
		\end{array}
		\right).
		$$
		Then the function
		$$f(x) = h_1(u_1^\tT x, u_2^\tT x ) + h_2 (u_1^\tT x, u_3^\tT x), \quad x=(x_1,x_2,x_3,x_4)$$
		has in general no sparse decomposition in the components of $x$, 
		but
		$$f_U(x) = f(Ux) = h_1(u_1^\tT U x, u_2^\tT U x ) + h_2 (u_1^\tT U x, u_3^\tT Ux) =
		h_1(x_1,x_2) + h_2(x_1,x_3)$$
		admits such a decomposition. This phenomenon is also illustrated in Figures \ref{fig:f_1_rot},\ref{fig:plt_f2},\ref{fig:plt_f2_rot}.
	\end{example}
	
	\begin{figure}
		\centering
		\begin{minipage}{.49\linewidth}
			\includegraphics[width=\linewidth]{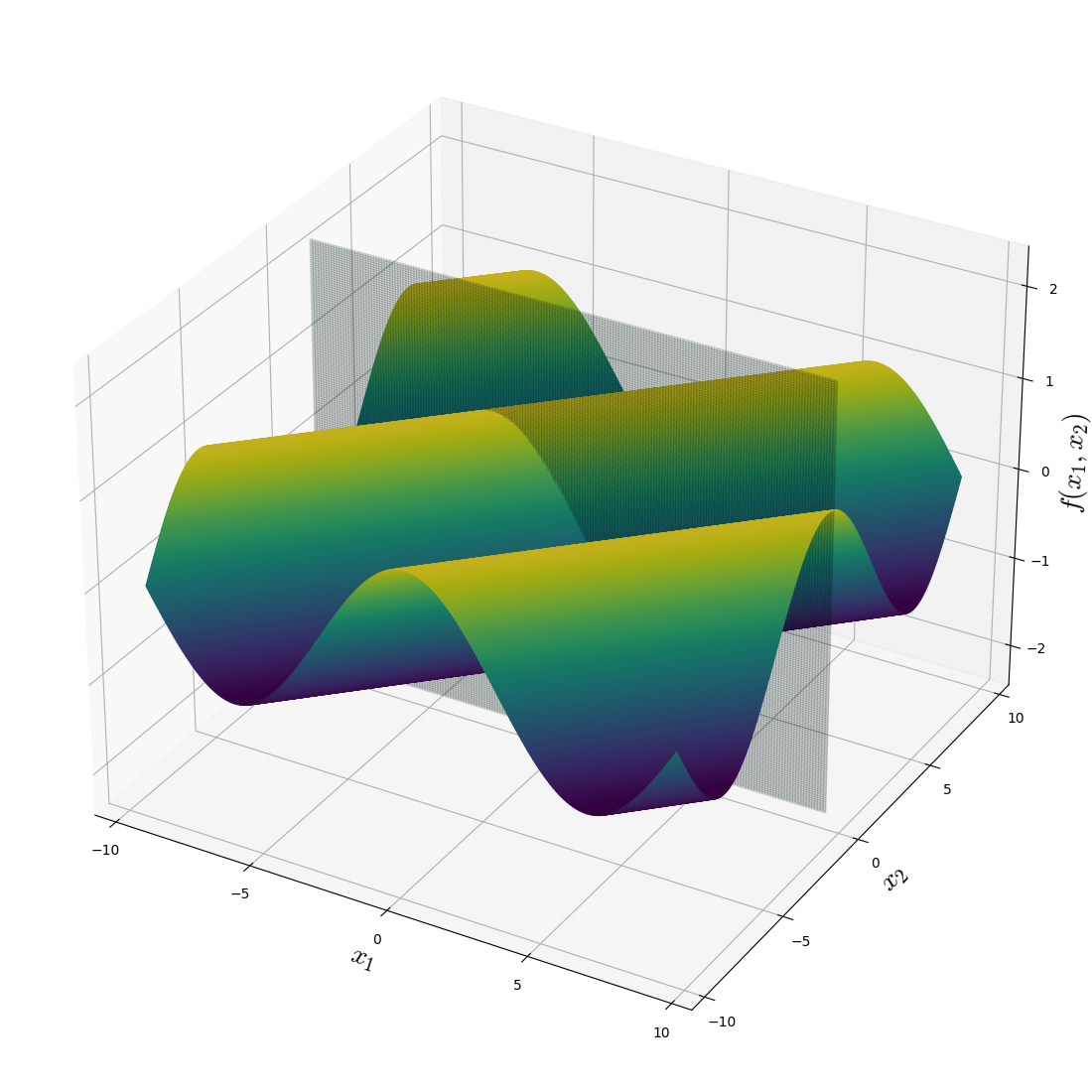}
			%\caption{$L_{\infty}$}
			%\label{fig:enter-label}
		\end{minipage}
		\begin{minipage}{.49\linewidth}
			\includegraphics[width=\linewidth]{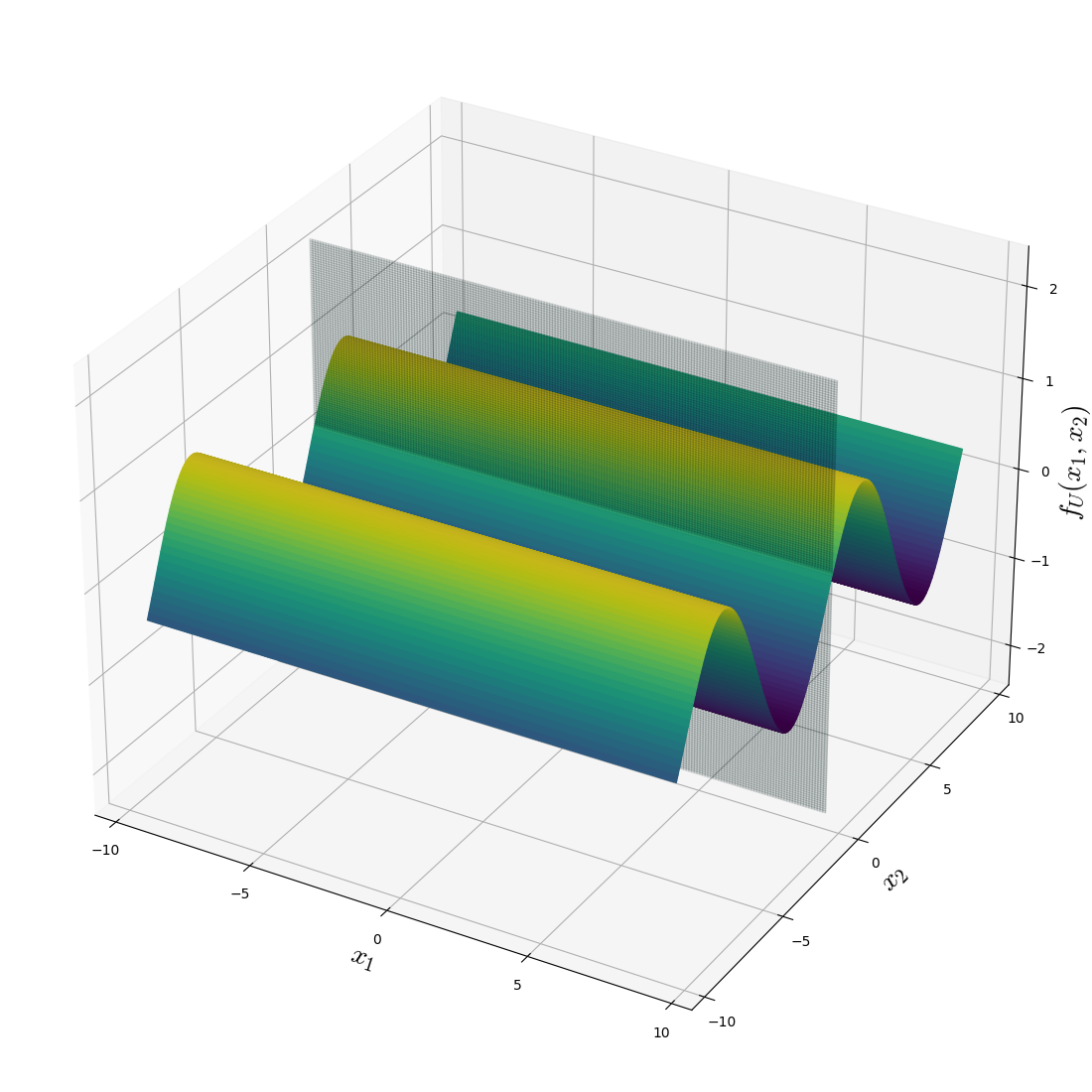}
			%\caption{$L1-norm$}
			%\label{fig:enter-label}
		\end{minipage}
		\caption{Left: $f(x) \coloneqq \sin(u_1^{\tT}x\sqrt{2}/2)$ where $U=(u_1,u_2)$ is a $2$-dimensional rotation matrix of rotation angle $\pi/4.$ Right: 
			Plot of $f_{U}.$ The left-hand diagram shows that $f$ depends on both variables $x_1$ and $x_2$ while $f_U$ is constant in $x_1.$}
		\label{fig:f_1_rot}
	\end{figure}
	
	\begin{figure}[h!]
		\centering
		\begin{minipage}{.49\linewidth}
			\includegraphics[width=\linewidth]{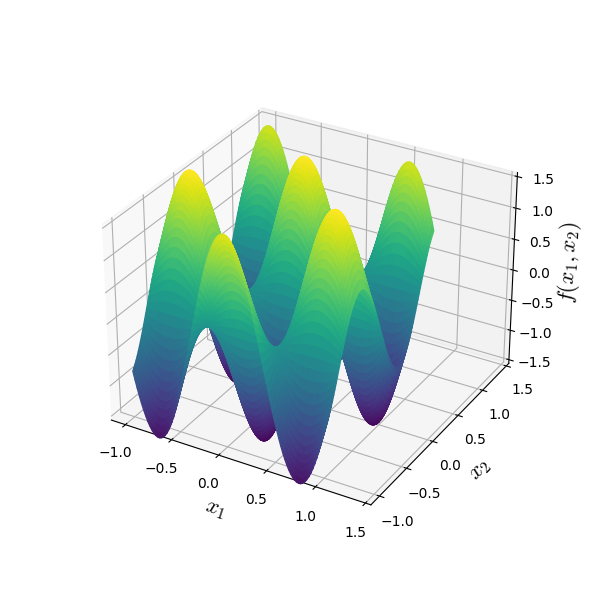}
			%\caption{$L_{\infty}$}
			%\label{fig:enter-label}
		\end{minipage}
		\begin{minipage}{.49\linewidth}
			\includegraphics[width=.7\linewidth]{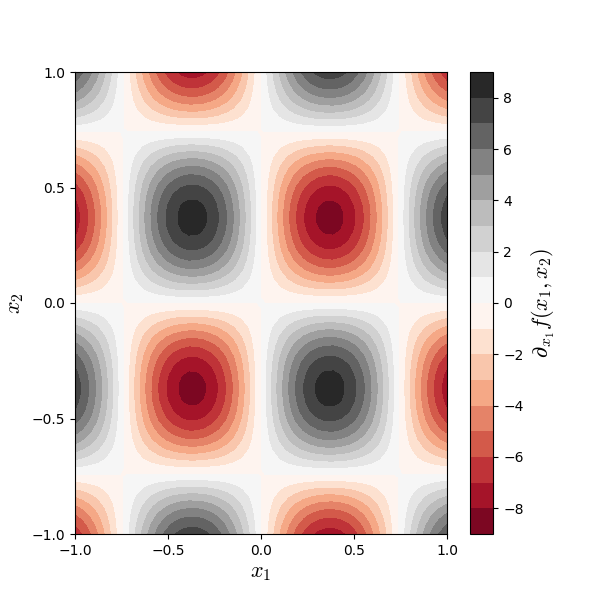}
			%\caption{$L1-norm$}
			%\label{fig:enter-label}
		\end{minipage}
		\caption{ Left: $f(x)= \sin(5 u_1^{\tT}x) + \sin(5 u_2^{\tT}x)$ where $U=(u_1, u_2)$ is a $2$-dimensional rotation matrix of angle $\pi/4.$ Right: Image of the partial derivative $\partial_{x_1}f(x).$ The partial derivative $\partial_{x_1}f(x)$ depends on $x_1$ and $x_2$ since its values vary in both variables and thus $f$ cannot be decomposed as a sum of two univariate functions.}
		\label{fig:plt_f2}
	\end{figure}
	
	Based on the given values of the gradient 
	$\nabla f (x^{(n)})$ 
	and the Hessian 
	$\nabla^2 f (x^{(n)})$ at $x^{(n)} \in \mathbb B_r (d)$,
	$n \in [N]$, 
	we are searching for those matrices 
	$U \in O(d)$ 
	for which the graph 
	$\mathcal G( f_U )$
	has the smallest number of vertices and edges.
	We will find (an approximation of) such $U \in O(d)$ in two steps which are detailed in the following subsections:
	
	\begin{figure}[h!]
		\centering
		\begin{minipage}{.49\linewidth}
			\includegraphics[width=\linewidth]{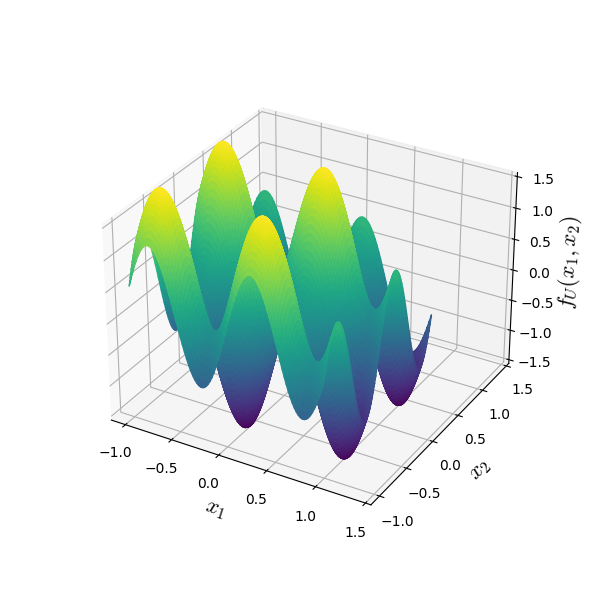}
		\end{minipage}
		\begin{minipage}{.49\linewidth}
			\includegraphics[width=.7\linewidth]{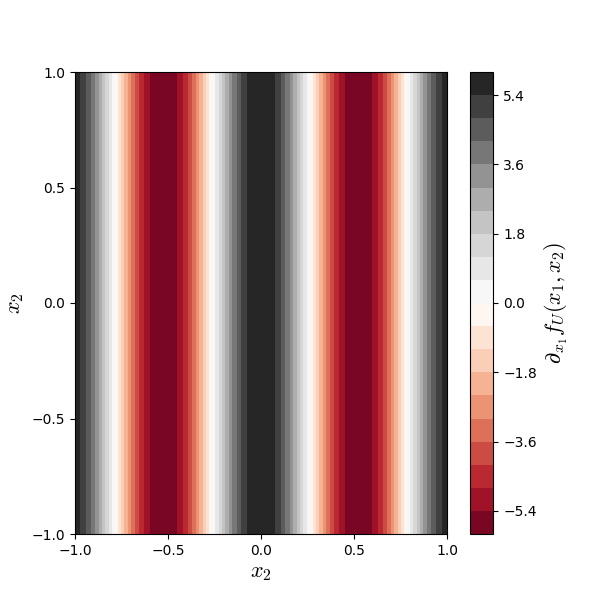}
		\end{minipage}
		\caption{Left: $f_U$ where $f, U$ are as in Figure~\ref{fig:plt_f2}. Right: Image of the partial derivative $\partial_{x_1}f_{U}$ which is constant with respect to $x_2$. Therefore $f_U$ can be decomposed as a sum of two univariate functions (see Theorem~\ref{decomposition:partial}).}
		\label{fig:plt_f2_rot}
	\end{figure}
	\begin{itemize}
		\item[1.]\label{schritt:1} \textbf{Vertex minimization}: Find a vertex minimizing matrix
		\begin{equation} \label{eq_v}
			U_{\mathcal V} \in \argmin_{U \in O(d)} \mathcal V (f_U).
		\end{equation}  
		\item[2.] \textbf{Edge minimization}: Using only the relevant vertices appearing in  $\mathcal G (f_{U_{\mathcal V}})$,
		say wlog related to $x = (x_1,\ldots,x_{d_1})$,
		$d_1 \le d$,
		we reduce our attention to $g(x) \coloneqq f_{U_{\mathcal V}}(x,0)$, 
		$x \in \mathbb B_{r}(d_1)$
		and search for an edge-minimizing matrix, i.e.
		\begin{equation} \label{min_e}
			U_{\mathcal E} \in \argmin_{U \in O(d_1)} \mathcal E (g)
		\end{equation} 
		in two steps:
		\begin{itemize}
			\item[2.1.]\label{schritt:2.1}  \textbf{Finest connected component decomposition}: 
			Find $U \in O(d_1)$ such that $\mathcal G(g_U)$ provides the ''finest'' decomposition into connected components.  
			\item[2.2.]  \textbf{Sparse component decomposition}:
			For each of the connected components, find an orthogonal matrix that transforms a connected component into one with the smallest number of edges. 
		\end{itemize}
	\end{itemize}
	
	%--------------------------------------------------------------
	\subsection{Vertex minimization}\label{mv} 
	%--------------------------------------------------------------
	First, we deal with the minimization problem \eqref{eq_v}. 
	By \eqref{eq: graph_v}, this is equivalent to the fact that most of the 
	directional derivatives of $f$ in direction $u_j$, $j=1,\ldots,d$
	vanish, i.e.
	\begin{equation}\label{eq:di1}
		\partial_{j} f_U = u_j^\tT \nabla f (U \cdot) = 0.
	\end{equation}
	Let 
	\[
	V_{\nabla f} \coloneqq \spa\{\nabla f(x) :x\in \mathbb B_r\} \subseteq \R ^d.
	\]
	
	The following lemma allows us to describe vertices of the graph of $f_U$  in terms of the gradient of $f$ at a finite number of points.
	
	\begin{lemma}\label{finite}
		Let $f\in C^1(\mathbb B_r)$ and $U \in O(d)$ 
		and assume that $V_{\nabla f} = \spa \{b_1,\ldots,b_N\}$.
		Then $\partial_{j}f_U=0$ if and only if $(U^\tT b_n)_j=0$ for all $n\in [N]$. 
	\end{lemma}
	
	\begin{proof} 
		By \eqref{eq:di1}, we have that
		$\partial_{j} f_U=0$
		if and only if
		$u_j^\tT v = 0$ for all $b\in V_{\nabla f}$. This is the case
		if and only if
		$u_j^\tT b_n = (U^\tT b_n)_j = 0$ for all $n \in [N]$.
	\end{proof}
	
	By the following proposition, the vertex minimizing transform in \eqref{eq_v}
	can be obtained via singular value decomposition.
	
	\begin{proposition}\label{grad:svd}
		For $x^{(n)} \in \mathbb B_r$, $n \in [N]$, let
		$$B \coloneqq  \left(\nabla f(x^{(1)}),\ldots ,\nabla f(x^{(N)})\right) = (b_1, \ldots,b_N) \in \R^{d,N}$$ 
		such that 
		$V_{\nabla f} = \spa \{ b_1, \ldots,b_N\}$. 
		Then the left singular matrix $U \in O(d)$ of $B$ is a minimizer of \eqref{eq_v}. 
	\end{proposition}
	
	\begin{proof}
		By \eqref{eq:di1}, we have that $U \in O(d)$ is a minimizer of \eqref{eq_v} if and only if $U^\tT B$ has the largest number of zero rows. 
		This is, if and only of $U$ contains the largest number of columns which are orthogonal to all columns of $B$.
		This is exactly given by a left singular matrix of $B$.
	\end{proof}
	
	Once we have found a minimizer $U_{\mathcal V}$ by an SVD of $B$, say wlog.
	$$
	U_{\mathcal V} = 
	(\underbrace{u_1,\ldots,u_{d_1}}_{U_1}, 
	\underbrace{u_{d_1+1} \ldots u_d}_{ U_2}) \in O(d)
	$$ 
	such that 
	$\spa\{u_1,\ldots,u_{d_1}\} = \spa \{b_1,\ldots,b_N\}$ and $U_2^\tT B  = 0$,
	we know that $f_{U_{\mathcal V}}$ does only depend on the first $d_1$
	components, so that we can restrict our attention to
	$$
	g(\underbrace{x_1,\ldots,x_{d_1}}_{x_{[d_1]}}) \coloneqq f_{U_{\mathcal V}}(x_1,\ldots,x_{d_1},0\ldots,0) =  f_{U_{\mathcal V}}(x_1,\ldots,x_{d}).
	$$
	It follows immediately that 
	$$
	\nabla^2 f_{U_{\mathcal V}} (x) = 
	\begin{pmatrix}
		H(x_{[d_1]})&0\\
		0&0
	\end{pmatrix}
	\quad \text{and} \quad
	H(x_{[d_1]}) \coloneqq \nabla^2 g(x_{[d_1]})\in \R^{d_1 \times d_1}.
	$$
	Further, we have by \eqref{der} that
	$$
	\nabla^2 f_{U_{\mathcal V}} (U_{\mathcal V}^\tT x) = 
	U_{\mathcal V}^\tT \nabla^2 f (x) U_{\mathcal V} =
	\begin{pmatrix}
		H\left((U_{\mathcal V}^\tT x)_{[d_1]} \right)&0\\
		0&0
	\end{pmatrix}.
	$$
	Thus, given $\nabla^2 f (x^{(n)})$, $n \in [N]$, 
	we obtain the values $\nabla^2 g\left((U_{\mathcal V}^\tT x^{(n)})_{[d_1]}\right)$ by
	\begin{equation}\label{edge_trafo}
		U_{\mathcal V}^\tT \nabla^2 f (x^{(n)}) U_{\mathcal V} 
		=
		\begin{pmatrix}
			\nabla^2 g\left( (U_{\mathcal V}^\tT x^{(n)})_{[d_1]}\right)&0\\
			0&0
		\end{pmatrix}
		.
	\end{equation}
	
	\begin{remark}\label{u_v:unique}
		The minimizing matrix $U_{\mathcal V}$   
		is not unique. However, it follows immediately
		that for any other left singular matrix $U$
		of $B$ we would have
		$$
		\nabla^2 f_{U} (x) = 
		P^\tT \begin{pmatrix}
			W^\tT H(x_{[d_1]}) W&0\\
			0&0
		\end{pmatrix} P
		$$
		with some $W \in O(d_1)$ and a permutation matrix $P$.
		In particular, the choice of the singular matrix does not influence the results in the next subsection.
	\end{remark}
	%--------------------------------------------------------------
	\subsection{Edge minimization}\label{ev} 
	%--------------------------------------------------------------
	Next, we are interested in the minimization problem \eqref{min_e} taking only the $d_1$ relevant variables 
	from the previous subsection into account.
	For simplicity of notation, we reset  $d_1 \to d$ and $(U_{\mathcal V}^\tT x^{(n)})_{[d_1]} \to x^{(n)}$, but keep the $g$. 
	By the previous section, we can assume  that $\nabla^2 g( x^{(n)} )$, $n \in [N]$ are given.
	By \eqref{eq: graph_v}, the edge minimization problem is equivalent to the fact that 
	\begin{equation}\label{eq:di1}
		\partial_{i,j} g_U = \partial_{j,i} g_U
		=
		u_i^\tT \nabla^2 g (U \cdot) u_j = 0
	\end{equation}
	holds true for the largest number of pairs $\{i,j\}$.
	Let 
	\[
	V_{\nabla^2 g} \coloneqq \spa\{\nabla^2 g(x) :x\in \mathbb B_r\} .
	\] 
	Clearly, the Hessians of $g$ at $x \in \mathbb R^{d}$ belong to the space 
	of symmetric $d \times d$ has therefore a spectral (eigenvalue) decomposition.
	Then, similarly as in Lemma \ref{finite}, we observe the following relation.
	%----------------------------------
	\begin{lemma}\label{finite Hessian}
		Let $g\in C^2(\mathbb B_r)$ and $U\in O(d)$ and assume that
		$
		V_{\nabla^2 g} = \spa\{H_1,\ldots,H_N\}
		$.
		Then $\partial_{i,j} g_U = 0$ if and only if $(U^T H_n U)_{ij}=0$ for all $n \in [N]$.
	\end{lemma}
	
	Now we proceed in two steps.
	
	\subsubsection{Finest connected component decomposition}
	We start with the definition of graphs with ,,finest connected components''.
	
	\begin{definition}[Graph with finest components]\label{preorder}
		For $d \in \N$, we denote by $\mathcal{N}_d$ 
		the set of all descending sequences 
		$a = (a_1,\ldots,a_{d_a})\in \N_{>0}^{k_a}$ satisfying $\sum_{i=1}^{d_a}a_i=d$.
		On $\mathcal{N}_d$ we introduce a preordering by 
		$a \preceq b$ if there exists a partition $\cup_i I_i=[d_a]$ such that
		for all $i=1,\ldots,d_b$ it holds
		\[
		b_i = \sum_{j \in I_i} a_j.
		\]
		On the set $G_d$ of all graphs with $d$ vertices, we define a function $\phi:G_d \to \mathcal{N}_d$ as follows: for a graph $\mathcal G \in G_d$,
		let $\mathcal G = \cup_{i=1}^{K} \mathcal G_i$ be its finest decomposition into disjoint connected components, i.e., the $\mathcal G_i$ cannot be further decomposed into disjoint connected components, and define 
		$\phi(\mathcal G) \coloneqq (|\VV(\mathcal G_1)|,\ldots,|\VV(\mathcal G_{K}|)\in\mathcal{N}_d$. 
		Then a preordering on $G_d$
		is given 
		by $\mathcal G \preceq \tilde {\mathcal G}$ if and only if 
		$\phi(\mathcal G)\preceq \phi(\tilde {\mathcal G})$.
		Now we consider the subset
		$\mathcal S \subset G_d$
		and define the set of \textit{graphs with finest connected components} in $\mathcal S$
		to be its minimal elements with respect to $\preceq$.
	\end{definition}

	We are interested in the minimal elements of
	$$\{\mathcal G(g_U): U \in O(d)\} \subset G_d,$$
	i.e., we are asking for $U \in O(d)$ such that $\mathcal G(g_U)$ is a graph with finest connected components.
	Assuming that 
	$$V_{\nabla^2 g} = \spa\{H_n \coloneqq \nabla^2 g( x^{(n)} ): n \in [N]\},$$
	we see by Definition \ref{def:graph}, that this is equivalent to
	determining $U \in O(d)$ such that $H_n$, $n \in [N]$ admit the finest joint
	block diagonalization.
	
	\begin{definition}[Finest joint block diagonalization.]
		For $H_n \in \R^{d \times d}$, $n \in [N]$, let 
		\begin{equation}\label{eq: finest block diag}
			U^T H_n U = \mathrm{diag}(H_n^{U,1},\ldots, H_n^{U,K})
		\end{equation}
		denote a joint block diagonalization of the $H_n$, $n \in [N]$ into $K$ blocks $H_n^{U,k}\in \R^{d^U_k \times d^U_k}$ which cannot be further splitted.  
		Define
		$\phi:O(d)\to \mathcal{N}_d$ by
		$\phi(U)\coloneqq d_U=(d^U_1,\ldots,d_K^U)$. Then  $U\preceq \tilde{U}$ if and only if $\phi(U)\preceq \phi(\tilde{U})$ is a preordering on $O(d)$. We say that \eqref{eq: finest block diag} is the \textit{finest joint block diagonalization} if $U$ is minimal with respect to $\preceq$.
	\end{definition}
	
	\begin{remark} In \cite{blockdiag} the term finest joint block diagonalization is defined in terms of matrix $*$-algebras. It is shown Remark \ref{definitions:equiv} that both definitions are equivalent.
	\end{remark}
	
	To find such a finest block diagonalization, we apply a technique based on the following observation.
	
	\begin{proposition}[Joint block diagonalization]
		For   $c \in \mathbb R^N$ randomly sampled from the uniform distribution on the unit sphere $\mathbb S^{N-1} \in \R^N$, let 
		\begin{equation} \label{joint}
			H\coloneqq \sum_{n=1}^N c_n H_n.
		\end{equation}
		Let  $U \in O(d)$ be a matrix that diagonalizes  $H$.
		Then, with probability one, $U$ provides a finest joint block diagonalization of the $H_n$, $n \in [N]$. More precisely, 
		the set of all randomly sampled $c\in \mathbb S^{N-1}$ from the uniform distribution such that there exists $U \in O(d)$ that diagonalizes $H$
		but does not provides to a finest block diagonalization has measure zero.
	\end{proposition}
	
	For a proof see \cite[Proposition 3]{murota2010numerical}.
	By the proposition, we can just 
	look for the spectral decomposition of
	a matrix $H$ of the form \eqref{joint}.
	However, this strategy, is susceptible to noise. 
	Therefore, Maehara and Murota \cite{blockdiag}  suggested an extension using the space of matrices that commute with the $H_n$, $n \in [N]$.
	
	\begin{proposition}{\cite[Proposition 3.8]{blockdiag}} \label{genau}
		Let  $A_1,\ldots,A_M$ 
		be a basis of the matrix space
		\begin{equation}\label{commutant:algebra}
			\{ A \in  M_d(\R): [A, H_n] := A H_n - H_n A = 0 \text{ for all } n \in [N] \}.
		\end{equation}
		For a randomly drawn $c \in \mathbb S^{M-1}$ from the uniform distribution, let
		$U \in O(d)$ diagonalize 
		\begin{equation} \label{joint_a}
			A \coloneqq \sum_{m=1}^M c_m A_m.
		\end{equation}
		Then, with probability one $U$ provides a finest block diagonalization of $H_n$, $n \in [N]$.
	\end{proposition}
	
	An important observation in \cite{blockdiag} 
	is that if the condition in \eqref{commutant:algebra} 
	is relaxed to 
	\begin{equation} \label{relax}
		\|[A, H_n]\|_{\text{F}} \leq \delta \quad \text{for all } n \in [N]
	\end{equation}
	with small $\delta > 0$, 
	then a matrix $U\in O(d)$ that diagonalizes $A$, also 
	jointly block diagonalizes the
	$H_n$, $n \in [N]$ up to a small error.
	Here $\| \cdot \|_{\text{F}}$ denotes the Frobenius norm.
	More precisely,
	consider the linear operators $T_n:\R^{d \times d}\to \R^{d \times d}$ with
	$A \mapsto [A,H_n]$ and set 
	\begin{equation} \label{T}
		T \coloneqq \sum_{n=1}^N T_n^\tT T_n.
	\end{equation}
	With the columnwise vectorization operator 
	$\text{vec}: R^{d \times d} \to \R^{d^2}$
	and the Kronecker product $\otimes$ of matrices we have that 
	$$
	\text{vec} \left( T_n (A) \right)
	=  \zb T_n \text{vec}(A) 
	, \quad \zb T_n \coloneqq H_n^\tT \otimes I_d - I_d \otimes H_n
	$$
	and 
	$ \|T_n (A)\|_{\text{F}} = \| \zb T_n \text{vec}(A) \|_2$.
	Let $V_1,\ldots,V_K \in \R^{d \times d}$ be orthonormal eigenvectors of $T$ with eigenvalues $\lambda_k$ smaller than $\delta^2$.
	Then, we obtain for
	$$
	V \coloneqq \sum_{k=1}^K c_k V_k
	$$
	and $v \coloneqq \text{vec} (V)$ that
	\begin{align}
		\sum_{n=1}^N\|[V^\tT,H_n]\|_F^2
		&=\sum_{n=1}^N\|[V,H_n]\|_F^2
		=
		\sum_{n=1}^N\|T_n(V)\|_F^2
		=
		\sum_{n=1}^N\|\zb T_n v\|_2^2
		\\
		&=
		\sum_{n=1}^N \langle \zb T^\tT_n \zb T_n v,v\rangle 
		= 
		\langle \zb T v ,v \rangle 
		=
		\sum_{k=1}^K c^2_k \lambda_k 
		\leq
		\delta^2
		\sum_{k=1}^K c^2_k.
	\end{align}
	Thus, for any $c\in \mathbb S^{K-1}$, the matrix 
	$\tfrac{1}{2}(V+V^\tT)$ 
	fulfills \eqref{relax}.
	Then we know by \cite[Lemma 4.1]{blockdiag}, for 
	$U\in O(n)$ satisfying 
	$\tfrac{1}{2}U^\tT (V+V^\tT) U =\diag(\lambda_1,\ldots,\lambda_d)$,  
	that
	\[
	(U^\tT H_n U)_{ij}|\lambda_i-\lambda_j|\leq \delta.
	\]
	Consequently, $U^\tT H_n U$ has an almost block diagonal structure 
	and the blocks correspond to different eigenvalues. 
	The resulting error-controlled version of the block diagonalization is stated as Algorithm \ref{alg: block diag}.
	
	\begin{algorithm}
		\begin{algorithmic}
			\State\textbf{Input:} 
			Hessian matrices $H_n$, $n \in [N]$, 
			error tolerance $\delta>0$.
			\State Find an orthonormal basis $V_1,\ldots,V_K$ corresponding to eigenvalues smaller than $\delta^2$ of $T$ in \eqref{T}.
			
			\State Sample $c \in \mathbb S^{K-1}$ randomly from the uniform distribution on $\mathbb S^{K-1}$
			and set $V \coloneqq \sum_{k=1}^K c_k V_k$.
			
			\State Compute $U\in O(d)$ that diagonalizes $\tfrac12 (V+V^\tT)$. 
			
			\State \textbf{Output: $U$} 
		\end{algorithmic}
		\caption{Error-controlled block diagonalization }\label{alg: block diag}
	\end{algorithm}
	
	While we know by Proposition \ref{genau} 
	that the above algorithm is guaranteed to find the finest block diagonalization with probability $1$ for $\delta=0$, 
	we did not find a uniqueness statement in the literature. 
	The following theorem contains the desired result.
	Its proof requires a deeper look into theory of matrix-$*$ algebras and can be found in Appendix \ref{appendix:block}.
	
	\begin{theorem}\label{uniqueness:block}
		Let $U_1,U_2$ correspond to finest joint block diagonalizations of $H_n$, $n \in [N]$ with block sizes $d_1^{U_1},\ldots,d_{K_1}^{U_1}$ and $d_1^{U_2},\ldots,d_{K_2}^{U_2}$ and
		$$
		U_1^T H_n U_1= {\text{blockdiag}}(H_n^{U_1,1},\ldots,H_{n}^{U_1,K_1}) , 
		\quad 
		U_2^T H_n U_2= {\text{blockdiag}}(H_n^{U_2,1},\ldots,H_{n}^{U_2,K_2}).
		$$
		Then, it holds $K_1 = K_2 =: K$. 
		Further, there exists a permutation 
		$\sigma$ of $[K]$ and matrices $V_k \in O(d_k^{U_1})$, $k \in [K]$ 
		such that for all $k \in [K]$ and all $n \in [N]$
		we have
		\begin{itemize}
			\item[ i)] $d_k^{U_1}=d_{\sigma(k)}^{U_2}$ 
			and
			\item[ii)] $H_{n}^{U_2,\sigma(k)}=V_k^\tT H_n^{U_1,k} V_k$.
		\end{itemize}
	\end{theorem}
	
	Note that neither $U_{\mathcal{V}}$ from \eqref{eq_v} nor $U$ obtained from the finest joint block diagonalization algorithm is unique. However, the only ambiguity is a possible conjugation of the blocks by orthogonal matrices as Remark \ref{u_v:unique} and Theorem \ref{uniqueness:block} show.

	%-------------------------------------------------
	\subsubsection{Sparse component decomposition}
	%--------------------------------------------------------------
	Having found a joint approximate finest block decomposition 
	of the Hessians $H_n$, $n \in [N]$, i.e., 
	\begin{equation} \label{eq:block}
		U^\tT H_n U = \text{blockdiag} (H_n^{U,1},\ldots,H_n^{U,K}), \quad 
		H_n^{U,k}  \in \R^{d_k^{U,k} \times d_k^{U,k}}
	\end{equation}
	it remains to to enforce the sparsity in each block.
	In particular, we can treat the blocks separately which reduces the dimension of the problem drastically.
	Therefore, we consider for an arbitrary $k \in [K]$, the $k$-th blocks 
	$B_n = H_n^{U,k}$
	in \eqref{eq:block}
	and set $d \coloneqq d_k^{U,k}$ now.
	Moreover, we agree that the diagonal elements of $B_n$, $n \in [N]$ are zero.
	The $0$-"norm" $\|B\|_0$
	of a $d \times d$ matrix $B$ is the number of its nonzero components.
	We would like to find a matrix $U \in O(d)$ that minimizes
	\begin{equation} \label{eq:l0}
		\ell(U) \coloneqq  \big\| \frac{1}{N} \sum_{n=1}^N (U^\tT B_n U)^2 \, \big \|_0
	\end{equation}
	where the square of the matrix is taken compomentwise.
	Unfortunately, it is well-known that the $0$-"norm" minimization is NP-hard. 
	Therefore, we will instead minimize the relaxed differntiable loss function
	\begin{align}\label{loss:smooth}
		\ell_\varepsilon(U) 
		= \sum_{\substack{i,j =1 \\ i\neq j} }^d \Big(\frac{1}{N}\sum_{n=1}^N (U^T B_n U)_{i,j}^2 +\varepsilon \Big)^\frac12,
		\quad
		\varepsilon>0.
	\end{align}
	The following proposition provides a sufficient condition that the relaxed version coincides with original one.
	
	\begin{proposition}\label{lemma:sufficient}
		Let $B_n \in \R^{d \times d}$, $n \in [N]$
		with zero diagonal components.
		If $U\in O(d)$ fulfills
		\[
		\ell_\varepsilon(U) = \max\Big\{\mathrm{dim}(\spa\{B_n, n \in [N] \}),\max_{n \in [N]} \rank(B_n)\Big\},
		\]
		then $U$ is a minimizer of \eqref{eq:l0}.
	\end{proposition}
	
	\begin{proof}
		By the properties of the rank, we get  for arbitrary $M_1, \ldots, M_N$ that
		\[
		\left\|\left(\|((M_1)_{ij},\ldots, (M_N)_{ij})\|_{\infty}\right)_{i,j = 1}^d \right\|_0 
		\geq \| M_n \|_0 \ge \rank(M_n), \quad n \in [N]
		\]
		and consequently
		\[
		\left\|\left(\|((M_1)_{ij},\ldots, (M_N)_{ij})\|_{\infty}\right)_{i,j = 1}^d \right\|_0 
		\geq \max_{n \in [N]} \rank(M_n).
		\]
		Now let $\mathcal B :=\{e_{i,j}:\exists M_k \text{ such that } (M_n)_{i,j} \neq 0\}$. 
		Then, 
		$\spa(\mathcal B) \supseteq \spa\{M_n, n \in [N]\}$ 
		and 
		$\dim(\spa(\mathcal B))=\left\|\left(\|((M_1)_{ij},\ldots ,(M_N)_{ij})\|_{\infty}\right)_{i,j\in[d]^2}\right\|_0$, 
		which gives
		\[
		\left\|\left(\|((M_1)_{ij},\ldots (M_N)_{ij})\|_{\infty}\right)_{i,j\in[d]^2}\right\|_0 \geq \mathrm{dim}(\spa\{M_n, n \in [N]\}).
		\]
		Therefore, for $M_n \coloneqq U^T B_n U$ the above inequalities yield
		\begin{align}
			\ell_\varepsilon(U) 
			& = \left\| \left(\|((M_1)_{ij},\ldots (M_N)_{ij})\|_{\infty}\right)_{i,j=1}^d \right\|_0 \\ 
			& \ge \max\Big\{\mathrm{dim}(\spa\{M_n , n \in [N]\}),\max_{n \in [N]} \rank(M_n)\Big\}. 
		\end{align}
		Since $U \in O(d)$, it holds that 
		$\mathrm{dim}(\spa\{B_n, n \in [N]\})=\mathrm{dim}(\spa\{U^T B_n U, n \in [N]\})$ 
		and 
		$\rank(B_n)=\rank(U^T B_n U)$ for all $n \in [N]$.
		Thus, we obtained a constant lower bound for $\ell_{0,e}(U)$,
		\[
		\ell_\varepsilon(U) \ge \max\Big\{\mathrm{dim}(\spa\{B_n, n \in [N] \}),\max_{n \in [N]} \rank(B_n)\Big \}
		\]
		and if it is reached, $U$ has to be the global minimizer of $\ell_\varepsilon$.
	\end{proof}
	
	%--------------------------------------------------------
	\section{Minimization of the loss $\ell_\varepsilon$ over $SO(d)$}\label{sec:MAnOpt}
	%-------------------------------------------------------
	This section deals with the efficient minimization of $\ell_\varepsilon$ in \eqref{loss:smooth}.
	In  Subsection \ref{subsec41},  we  propose a gradient descent algorithm on 
	$SO(d)$ and also consider the Landing algorithm
	for a regularized version of $\ell_\varepsilon$.
	Subsection \ref{subsec42} contains convergence results.
	The efficient computation requires further a grid search to have an appropriate
	starting point for the algorithm.
	The corresponding results can be found in Appendix \ref{sec43},
	where we also have a closer look to the computational complexity.
	
	%---------------------------------------------------------
	\subsection{Gradient descent and Landing algorithm over $SO(d)$} \label{subsec41}
	If $U \in O(d)$ is a minimizer of \eqref{loss:smooth}, then
	$\text{blockdiag}(-1,I_{d-1} ) U \in SO(d)$ is a minimizer as well so that the optimization can be reduced 
	to those over $SO(d)$.
	Now, we turn to the minimization of the loss function \eqref{loss:smooth} on $SO(d)$ using gradient-based manifold optimization techniques.
	Let us start with a short overview of the basic definitions and properties related to $SO(d)$ as a manifold. For further details, we refer to \cite{ablin2022fast, boumal2023intromanifolds, hu2020brief}. Recall that $SO(d)$ is a Riemannian manifold.  For each $U \in SO(d)$ the tangent space to $SO(d)$ at $U$ is given by 
	\[
	T_U = \{ V \in \R^{d \times d}: V U^T + U V^T = 0\} = \{ A U : A \in \R^{d \times d}, A + A^T = 0 \}.
	\]
	For a function $F: \R^{d \times d} \to \R$ differentiable in the neighborhood of $SO(d)$, its Riemannian gradient $\grad F(U)$ at point $U$ is given by an orthogonal projection of $\nabla F(U)$ onto tangent space $T_U$. For $SO(d)$ specifically, it has a closed-form 
	\[
	\grad F(U) := \tfrac{1}{2} \nabla F(U) - \tfrac{1}{2} U \nabla F(U)^T U.
	\]
	We denote by $TSO(d)\coloneqq \cup_{U\in SO(d)}\{U\}\times T_U$ the tangent bundle of $SO(d)$. A retraction operator is a smooth map of manifolds $R:TSO(d)\to SO(d)$ which satisfies for all $U\in SO(d)$ that 
	\begin{itemize}
		\item $R(U,0)=U$,
		\item $DR_U(0)=\text{Id}_{T_U}$ where $R_U=R|_{\{U\}\times T_U }:T_U\to SO(d)$ and $D$ is the differential.
	\end{itemize}
	Note that the last properties ensures that for a line $\gamma(t)=tV$ in the tangent space $T_U$ we have that $\frac{d}{dt}R(U,\gamma(t))|_{t=0}=V$ i.e. a retraction is a first-order approximation of the exponential map.
	There are various choices of retraction operators for $SO(d)$ such as the exponential map, the Cayley transform, and the Polar decomposition \cite[Chapter~3.3]{hu2020brief}, see also \cite{parseval2020}. 
	In this work, we use the QR-factorization \cite{hu2020brief} based retraction operator defined as
	\begin{equation}
		\Retr(U, V) = \mathrm{QR}(U-V),
	\end{equation}
	where $\mathrm{QR}(U-V)$ denotes the orthogonal matrix of the QR-factorization of $U-V$.  
	
	Among many optimization methods on arbitrary Riemannian manifolds \cite{hu2020brief}, and $O(d)$ \cite{ablin2022fast} and $SO(d)$ \cite{Usevich.2020} in particular, we focus on the gradient descent and Landing methods.  
	Given an initial guess $U^{(0)}$ and step sizes $\{ \nu_r \}_{r \ge 0}$, the sequence $\{ U^{(r)} \}_{r \ge 0}$ constructed via Riemannian gradient descent is given by
	\begin{equation}\label{eq: Riemannian gradient descent}
		U^{(r+1)} = \Retr\left(U^{(r)}, - \nu_r \grad \ell_\varepsilon(U^{(r)})\right), \quad r \ge 0. 
	\end{equation}
	If the algorithm converges to a fixed point if $U^{(r+1)} = U^{(r)}$, the condition $\grad \ell_\varepsilon(U^{(r)}) = 0$ is satisfied. 
	Computing the retraction operator can be very time-consuming when the space dimension $d$ is large. This is why, for $O(d)$, an alternative method called the Landing algorithm was developed in ~\cite{ablin2022fast}. Instead of performing the minimization of $\ell_\varepsilon$ on $SO(d)$, it minimizes the penalized objective
	\[
	\ell_\varepsilon(U) +  \tfrac{\lambda}{4}\norm{ I_d  - U U^T}_F^2,
	\]
	over $\R^{d \times d}$ with the regularization parameter $\lambda > 0$. Note that the penalty is zero if and only if $U \in O(d)$. The Landing update step is given by
	\begin{equation}\label{eq: landing}
		U^{(r+1)}= U^{(r)} - \nu_r \left( \grad \ell_\varepsilon(U^{(r)}) +\lambda \left[ U^{(r)} (U^{(r)})^T - I_d \right] U^{(r)} \right), \quad r \ge 0,
	\end{equation}
	where the second term pushes $U^{(r+1)}$ in the direction of the manifold. Note that in general $U^{(r+1)} \notin O(d)$ and only after a certain number of iterations $U^{(r+1)}$ comes close to $O(d)$. Therefore, the fixed point $U$ of the Landing algorithm admits 
	\[
	\grad \ell(U) = - \lambda \left[ U U^T - I_d \right] U.
	\]
	If $U \in O(d)$, once again the gradient vanishes.
	
	%-----------------------------------------------------------------
	\subsection{Convergence analysis}\label{subsec42}
	In the previous subsection,  we observed that if the algorithms converge, the fixed point $U \in SO(d)$ admits $\grad(U) = 0$. Yet, it may not necessarily be global minima, unless the function $\ell_\varepsilon$ is geodesically convex \cite[Corollary 11.18]{boumal2023intromanifolds}. The latter is not true as the following theorem shows, see \cite{Shing-Tung.1974}.
	
	\begin{theorem}\label{lem:geo_convex}
		Let $F: SO(d) \to \R$ be a continuous geodesically convex function on $SO(d)$. Then, $F$ is constant.
	\end{theorem}
	
	Consequently, convergence to a global minimum of either of the methods depends on the initial guess $U^{(0)}$. In the following, we focus on the Riemannian gradient descent and derive its local sublinear convergence. As for the Landing algorithm, its local convergence to the global minimizer remains an open problem. We start with the following sublinear convergence result.  
	
	\begin{theorem}\label{thm: manifold optimization global}
		There exist $L>0$ such that the sequence $\{ U^{(r)}\}_{r \ge 0}$ generated by Riemannian gradient descent for \eqref{loss:smooth} with step sizes $\nu_r = \nu \le 1/L$ admits
		\begin{equation}\label{eq: sufficient decrease}
			\ell_\varepsilon(U^{(r+1)}) - \ell_\varepsilon(U^{(r)}) 
			\le - \tfrac{1}{2L} \norm{\grad \ell_\varepsilon(U^{(r)})}_F^2, \quad r \ge 0. 
		\end{equation}
		Furthermore, $\{ \ell_\varepsilon(U^{(r)}) \}_{r \ge 0}$ converges and $\grad \ell_\varepsilon(U^{(r)}) \to 0$ as $r \to \infty$. 
	\end{theorem}
	\begin{proof}
		We first note that $\ell_\varepsilon(U) > 0$ for all $U \in \mathbb R^{d \times d}$. Then, the convergence results follow from \eqref{eq: sufficient decrease} via \cite[Theorem 2.5]{Boumal.2018}. Hence, we only need to show that \eqref{eq: sufficient decrease} holds, which is, in turn, guaranteed by \cite[Lemma 2.7]{Boumal.2018}. Therefore, let us check that all the conditions of \cite[Lemma 2.7]{Boumal.2018} are satisfied. The manifold $SO(d)$ is a compact Riemannian submanifold of $\mathbb R^{d \times d}$.
		Furthermore, $\ell_\varepsilon$ is a continuously-differentiable function on a compact set, and, hence, it is Lipschitz-continuous \cite[Corollary 6.4.20]{Sohrab.2014}. 
	\end{proof}
	
	Theorem \ref{thm: manifold optimization global} only ensures convergence of manifold optimization to a fixed point and not necessarily to a global minimum, which is a typical outcome in nonconvex optimization. Let us define 
	\[
	\ell^* := \min_{U \in SO(d)} \ell_\varepsilon(U),
	\quad \mathcal{M} := \argmin_{U \in SO(d)} \ell_\varepsilon(U),
	\quad \text{and} \quad \mathcal{F} := \{U \in SO(d): \grad \ell_\varepsilon(U) = 0\}.
	\]
	The next result ensures that by initializing Riemannian gradient descent in a neighborhood of the global minima, it will converge to it.
	
	\begin{theorem}\label{thm: manifold optimization local}
		There exists a level set $\{ U \in SO(d): \ell_\varepsilon(U) \le \ell^* + q^*\}$ with $q^*>0$ that contains all global minimizers $\mathcal M$ and no other fixed points $\mathcal F$ of $\ell$. Consequently, if $U^{(0)} \in \{ U: \ell_\varepsilon(U) < \ell^* + q^*\}$, the sequence generated by Riemannian gradient descent with step sizes as in Theorem \ref{thm: manifold optimization global} converges to a global minimum of $\ell_\varepsilon$.
	\end{theorem}
	
	\begin{proof}
		The main idea of the proof is to show that there exists an open neighborhood of $\mathcal M$ such that it does not contain other critical points $\mathcal F \backslash \mathcal M$ or in other words that $\mathcal M$ is isolated from $\mathcal F \backslash \mathcal M$. Then, we will show that $\ell_\varepsilon$ on $\mathcal F \backslash \mathcal M$ is strictly larger than $\ell^*$ and it is possible to find suitable $q^*>0$. 
		
		The proof is based on the {\L}ojasiewicz inequality. We say that function $\ell_\varepsilon$ 
		satisfies {\L}ojasiewicz inequality at point $U\in SO(d)$ if there exists $\delta > 0$ 
		such that for all $V \in SO(d)$, $\norm{U - V}_F \le \delta$ the inequality
		\[
		\norm{\grad \ell_\varepsilon(V) }_F \ge c |\ell_\varepsilon(U) - \ell_\varepsilon(V)|^{1-\zeta}
		\]
		holds for some $c>0$ and $\zeta \in [0,1/2)$. By \cite[Proposition 2.2 and Remark 1]{Schneider.2015}, for an analytic function on an analytic manifold\footnote{According to Definition 2.7.1 in \cite{Krantz.2002}}, {\L}ojasiewicz inequality is satisfied at every point on the manifold. $SO(d)$ is analytic and $\ell_\varepsilon \in C^{\infty}(r)$ is analytic as a superposition of a square root and a positive polynomial. Consequently, for each $U \in \mathcal M$ we can find an $\delta = \delta(U) >0$ 
		from {\L}ojasiewicz inequality. 
		
		Next, we show by contradiction that there exists a threshold $q^*>0$ such that $\ell_\varepsilon(U) > \ell^* + q^*$ on $\mathcal F \backslash \mathcal M$. Assume that the opposite holds and for every $q > 0$ there exists $U = U(q) \in \mathcal F \backslash \mathcal M$ such that $\ell_\varepsilon(U(q)) \le \ell^* + q$. Then, let us consider a sequence $\{U(1/k)\}_{k\ge 1} \subseteq \mathcal F \backslash \mathcal M$. Since $\mathcal F \backslash \mathcal M \subseteq SO(d)$ and $SO(d)$ is compact, $\{U(1/k))\}_{k\ge 1}$ is bounded and there exists a convergent subsequence $\{U(1/k_j)\}_{j \ge 1}$ with limit $U^*$. By construction, $U^* \in \mathcal F$ and it admits 
		\[
		\ell^* \le \ell_\varepsilon(U^*) 
		= \ell_\varepsilon\left( \lim_{j \to \infty}U(1/k_j) \right) 
		= \lim_{j \to \infty} \ell_\varepsilon\left( U(1/k_j) \right)
		\le \lim_{j \to \infty} \ell^* + \frac{1}{k_j}
		= \ell^*,
		\]
		so that $U^* \in \mathcal M$. However, from the convergence it follows that there exists $j_0 \in \mathbb N$ such that for all $j \ge j_0$ we have $\norm{U^* - U(1/k_j)}_F < \delta(U^*)$. The {\L}ojasiewicz inequality then gives
		\[
		0 = \norm{\grad \ell_\varepsilon(U(1/k_j)) }_F 
		\ge c |\ell_\varepsilon(U^*) - \ell_\varepsilon(U(1/k_j))|^{1-\zeta} \ge 0,
		\]
		which is only possible if $U(1/k_j) \in \mathcal M$ and $\ell_\varepsilon(U(1/k_j)) = \ell^*$.
		Yet, it contradicts $U(1/k_j) \in \mathcal C \backslash \mathcal M$. Therefore, we obtain the contradiction. We also note that the scenario of a global minimum at infinity is impossible as $SO(d)$ is compact. 
		
		Thus, there exists $q^*>0$ such that $\ell_\varepsilon(U) > \ell^* + q^*$ for all $U \in \mathcal F \backslash \mathcal M$. Consequently, all fixed points in the set $\mathcal L := \{U \in SO(d): \ell_\varepsilon(U) \le \ell^* + q^*\}$ are global minimizers of $\ell_\varepsilon$. If $U^{(0)} \in \mathcal L$, by Theorem \ref{thm: manifold optimization global}, the sequence $\{ U^{(r)}\}_{r \ge 0}$ generated by Riemannian gradient descent will remain in $\mathcal L$ and converge to a point in $\mathcal L \cap \mathcal F = \mathcal M$.     
	\end{proof}
	
	Theorem \ref{thm: manifold optimization local} has a flavor of standard convergence results based on {\L}ojasiewicz-type inequalities, e.g., \cite[Theorem 2.2]{Absil.2005} or \cite[Theorem 3.3]{Attouch.2010}. Although commonly a linear convergence is expected in a neighborhood of the accumulation point, a stronger assumption on the function is required. For instance, the exponent $\zeta$ in the {\L}ojasiewicz inequality has to be $1/2$, which is known as the Polyak-{\L}ojasiewicz inequality. For $C^2$ functions, it is equivalent to a number of other well-known conditions \cite{Rebjock.2023}. We refer the interested reader to Section 1.3 of \cite{Rebjock.2023} for an extensive literature overview on the topic. While there are some rules on the computation of {\L}ojasiewicz exponents \cite{Li.2018, Yu.2022}, the establishment of local linear convergence in the case of \eqref{loss:smooth} remains a topic of future research.     
	%------------------------------------------
	\section{Numerical results}\label{sect:numerics}
	In the following section, we will first investigate the performance of the manifold optimization method for sparsifying a set of symmetric matrices. Furthermore by applying the three-step algorithm consisting of vertex minimization, finest connected component decomposition and sparse component decomposition in order to find an optimal $U\in SO(d)$ for a set of functions, we demonstrate that the gradients and Hessians admit the optimal sparsity patterns. For two test functions, we show that our algorithm finds $U\in SO(d)$ such that $f_U$ exhibits the correct sparse ANOVA decomposition.
	\footnote{The code for our examples is available at https://github.com/fatima0111/Sparse-Function-Decomposition-via-Orthogonal-Transformation.}
	
	To run the numerical experiments we have used the following libraries: \textit{pytorch}~\cite{NEURIPS2019_9015}, \textit{numpy}~\cite{harris2020array} for the general computations, \textit{RiemannianSGD} from \textit{geoopt}~\cite{geoopt2020kochurov} to perform the Riemannian gradient descent and \textit{LandingSGD} from \cite{ablin2022fast} to execute the Landing procedure. To compute the ANOVA-terms in section~\ref{subsec:test_functions}, we have used the \textit{tntorch} library \cite{tntorch22}. All computations were performed on a NVIDIA RTX A$6000$ $48$GB graphics card.
	
	\subsection{Manifold optimization on $SO(d)$ for jointly sparsifying a set of symmetric matrices}\label{subsec:numeric_random_hessian}
	\subsubsection*{Creating jointly sparsifiable symmetric matrices}
	In order to test the feasibility of the manifold optimization methods for jointly sparsifying matrices we create sets of nonsparse matrices where we know that they can be made jointly sparse through conjugation with an orthogonal matrix. Let $J \subseteq [d]\times [d]$ be a set of jointly nonsparse entries. The matrices $\tilde H_n$, $1 \le n \le N$, are constructed by sampling $(\tilde H_n)_{i,j}=(\tilde H_n)_{j,i} \sim \textrm{Unif}[-1,1]$ whenever $(i,j) \in J$ and setting the rest of the entries to zero. Then, with randomly drawn $R\in SO(d)$, we take 
	\begin{equation}
		\mathcal{H}_R(J):= \set{H_{n}:=R^{T} \tilde H_n R: 1 \le n \le N}
	\end{equation}
	as input data for edge minimization.
	% Let $J\subseteq [d]\times [d]$ be a set satisfying that $(i,j)\in J$ implies $i\leq j$. We construct a randomly chosen set of $n\in\N$ symmetric matrices that are nonzero only at $\{(i,j),(j,i)\}\cap J=\emptyset$ as follows.
	% \begin{itemize}
		% \item Choose symmetric matrices $\tilde{H}_k=B_k+B_k^T$ for $B_k \in [-1, 1]^{d \times d}$ randomly drawn and $1\leq k\leq n$.
		% \item Create the set of sparse matrices $\mathcal{H}(J):=\{H_k:1\leq k\leq n\}$ where
		% \begin{align}
			%     (H_k)_{i,j} := \begin{cases}
				%     (\tilde{H}_k)_{i,j},\,& \text{if }(i,j)\in J\text{ or } (j,i) \in J,\\
				%     0,\,&\text{otherwise}.
				%     \end{cases}
			% \end{align}
		% \item Draw $R\in O(d)$ randomly and define
		% \begin{equation}
			%     \mathcal{H}_R(J):= \set{H_R:=R^{T}HR: H \in \mathcal{H}(J)}.
			% \end{equation}
		% \end{itemize}
	\begin{remark}
		For generically chosen $\tilde{H}_n$ and $R$ we have that $\dim\spa(\mathcal{H}_R(J))=\dim\spa(\mathcal{H}(J))=|J|$. Thus by Lemma \ref{lemma:sufficient} we know that for any minimizer $U\in SO(d)$ of $\ell_\varepsilon$ it holds that $U^T\mathcal{H}_R(J)U$ has exactly $|J|$ nonzero entries. 
	\end{remark}
	\textit{Creating noisy data.} Additionally, to investigate whether the manifold optimization procedures are robust towards noise, we define the set 
	\begin{align}
		\mathcal{H}_{R}(J,\sigma) \coloneqq \{H+\zeta_H:H\in \mathcal{H}_R(J)\}
	\end{align}
	where each entry $(\zeta_H)_{i,j}\sim \mathcal{N}(0,\sigma^2)$ is drawn from a Gaussian distribution with mean $0$ and variance $\sigma^2$. 
	
	\subsection*{Details of the manifold optimization}
	\textit{Reducing complexity.}
	To reduce the time complexity of the algorithm an approximate basis of the $\spa(\mathcal{H}_R(J))$ resp. $\spa(\mathcal{H}_R(J,\sigma))$  is computed via SVD. Choose a threshold $\tau$ and use SVD on $\mathcal{H}_R(J)$ resp. $\mathcal{H}_{R}(J,\sigma)$ to find all orthonormal vectors corresponding to singular values $s\geq \tau$. We denote the set of these orthonormal vectors by $H_R(J)$ resp. $H_{R}(J,\sigma)$. Note that even for $\mathcal{H}_R(J)$ the threshold $\tau$ is necessary since numerical errors may occur.
	Then, for $\mathcal H= \mathcal H_R(J)$ or $\mathcal H= \mathcal H_{R}(J,\sigma)$, we minimize
	\begin{align}\label{loss:h}
		\ell_{\mathcal H}(U)\coloneqq\frac{1}{\sqrt{\abs{\mathcal H}}}\sum_{(i,j)\in[d]^2}\Big( \sum_{H \in \mathcal H}(U^T H U)_{i,j}^2 + \varepsilon \big)^\frac12.
	\end{align}
	Note that we included the diagonal entries as they provide information on whether the function depends only linearly on the variables. \\[1ex]
	\textit{Random initialization.} Since the loss function is not convex, an appropriate initialization is needed. The random initialization method consists of generating randomly $5$ angles $\alpha_1, \cdots, \alpha_5 \in \Theta_d$ from the uniform distribution and to compute the corresponding rotation matrices $R_{\alpha_k}, k=1,\cdots, 5$ according to \eqref{eq: rotation parametrization}. Each rotation matrix will be used to initialize the Riemannian gradient descent or Landing method for only $5\cdot 10^3$ iterations to obtain the matrices $\hat{R}_{\alpha_k}$. Let 
	\begin{equation}
		\ell_{\frac 1 2, 2}\bigl(U\bigr) := \frac{1}{\sqrt{\abs{\mathcal H}}}\Big(\sum_{i,j=1}^ d \big(\sum_{H \in \mathcal H} (U^T H U)_{i,j}^2\big)^\frac14\Big)^2
	\end{equation}
	and use as the \emph{random initializer}
	\begin{align}
		R_{RI}^0:=\argmin_{1\leq k\leq 5}\ell_{\frac 1 2,2}\left(\hat{R}_{\alpha_k}\right).
	\end{align}
	Note that we used $\ell_{\frac 1 2,2}$ here because it better approximates the sparsity norm $\|\cdot\|_{0,\infty}$.\\
	\textit{Grid search initialization.}
	For the grid search in Section \ref{subsec:grid_search} and for $h\in\left\{1,0.5,0.25,0.125,0.1\right\}$ and $\Gamma(h)$ a grid as in \eqref{eq: grid set} we let
	\begin{align}
		R_h^0:=\argmin_{R\in\Gamma(h)}\ell_{\frac 1 2,2}(R).
	\end{align}
	For $d=5$ we only use this method for $h=1$ due to high computational complexity.
	
	\subsubsection*{Convergence and performance of the manifold optimization methods} 
	For $d=2,3,4,5$ we applied the manifold optimization methods for minimizing $\ell_{\mathcal H}$, see \eqref{loss:h}. For every dimension $d$ we created $100$ sets of jointly sparsifiable
	symmetric matrices $\mathcal{H}_R(J)$ resp.
	\begin{table}[h!]
		\centering
		\begin{tabular}{|c|c|c|c|c|}
			\cline{2-5}
			\multicolumn{1}{c|}{}& $d = 2$ & $d = 3$ & $d = 4$ & $d = 5$  \\
			\hline
			$1 \le |J| \le$ & $3$     & $6$     & $7$     & $11$ \\
			\hline
		\end{tabular}
		% \caption{Caption}
		\label{tab:size J}
	\end{table}
	$\mathcal{H}_R(J,\sigma)$ where $J$ was chosen randomly with the constraint that $\quad \quad$
	% \begin{equation}
		%    \abs{J} \in \begin{cases}
			%         \{1, \cdots, 3\},\, \text{if } d=2,\\
			%         \{1, \cdots, 6\},\, \text{if }  d=3,\\
			%         \{1, \cdots, 7\},\, \text{if }  d=4,\\
			%         \{1, \cdots, 11\},\, \text{if }  d=5.
			%     \end{cases}
		% \end{equation}
	\textit{Convergence.}
	Note that the minimum of our objective \eqref{loss:h} is unknown but we know that $R^T\mathcal{H}_R(J)R$ admits the optimal sparsity pattern. Let $U^{(r)}$ be the output of the manifold optimization at iteration $r$. We then compare $\ell_{\mathcal H}(U^{(r)})$ with $\ell_{\mathcal H}(R^T)$ Figure~\ref{fig:losses_random_hessians} in order to analyze the convergence of the manifold optimization methods. It shows that especially for $d\in\{4,5\}$ initialization with grid search admits a better convergence than random initialization which is in line with  Corollary \ref{cor:convergence_global_min}. The Landing method and the Riemannian gradient descent quantitatively give similar results both in objective value and runtime, see Figure~\ref{fig:losses_random_hessians} and Table~\ref{tab:runtime} in Appendix \ref{appendix:numerics}.
	
	\begin{figure}[ht!]
		\centering
		\includegraphics[width=\linewidth]{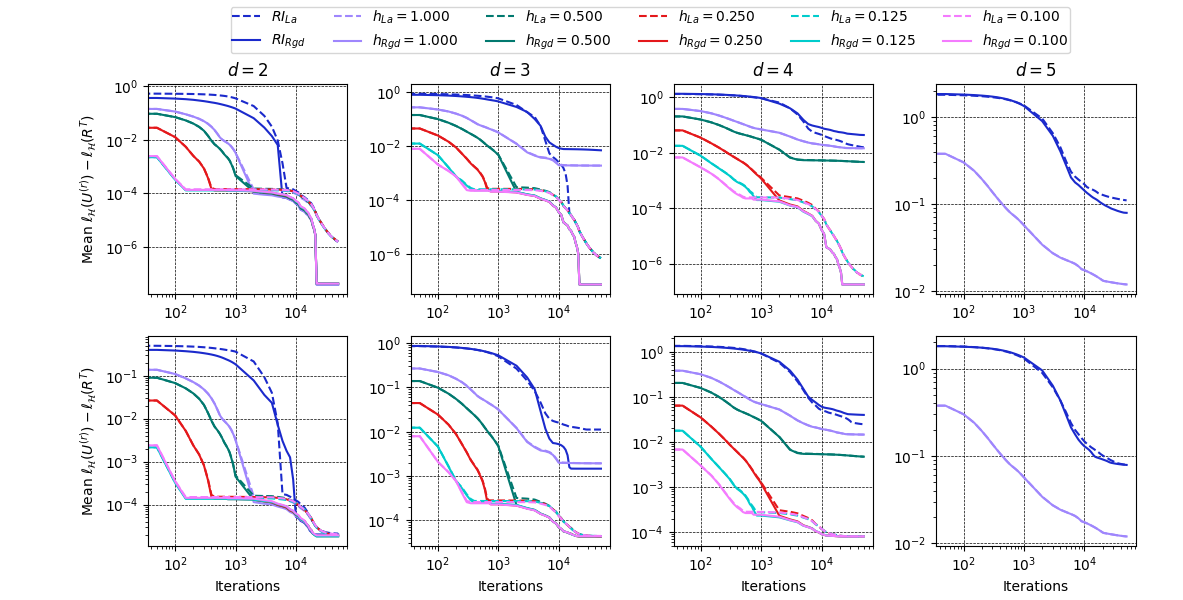}
		\caption{\small{Mean optimality gap $\ell_{\mathcal H}\bigl(U^{(r)}\bigr) - \ell_{\mathcal H}\bigl(R^{\tT}\bigr)$ over 100 experiments. Top: Noise-free matrices. Bottom: Noisy matrices. RI denotes random initialization and $h$ uses the grid search with the corresponding grid density value. Subscripts Rgd and La stand for Riemannian gradient descent and Landing algorithm, respectively.
				%Mean of the difference of the optimization method loss function $\ell^{(k)}=\ell_H(U_k)$ and the ground-truth loss $\ell_g=\ell_H(R^T)$ of $100$ sets of symmetric jointly sparsifiable matrices. 
				%RI\_La: Random initialization with landing algorithm, RI\_Re: Random initialization with Riemannian gradient descent.
		}}
		%h\_La=h: Grid search initialization with grid size $h$ and landing algorithm, h\_Re=h: Grid search initialization with grid size $h$ and retraction manifold optimization.}}
\label{fig:losses_random_hessians}
\end{figure}

\textit{Sparsifying performance.}
Since $\ell_\varepsilon$ approximates $\ell$, we will only obtain approximate sparsity and use an upper threshold instead of the $\|\cdot\|_{0,\infty}$ norm as follows. For $H\in\R^{d\times d}$, we define
$
|H|=(|H_{i,j}|)_{(i,j)\in[d]^2}
$
and 
\begin{align}\label{eq: thresholding}
\bar{H}:=\frac{1}{n}\sum_{H\in\mathcal{H}}|H|\in\R^{d\times d}, \quad
\left(\bar{H}_\eta\right)_{ij}:=\begin{cases}
	0,& \text{if } \bar{H}_{ij} \leq \eta,\\
	\bar{H}_{ij}, & \text{otherwise},
\end{cases}
\end{align}
where $\mathcal H =U^T\mathcal{H}_R(J)U$.
Thus we can compare the sparsity up to a level of $\eta$ by
\begin{equation}\label{difference_sparsity}
\chi(U, \eta) \coloneqq \norm{\bar{H}_\eta}_0 -|J|.
\end{equation}
The quantity $\chi(U,\eta)$ measures how far from the optimal sparsity $U^T\mathcal{H}_R(J)U$ is, details about the results can be found in Appendix \ref{appendix:numerics}. 

To highlight the impact of the thresholding parameter $\eta$ on the sparsity, we consider a failure ratio $\mathcal R$ as a mean
\begin{equation}\label{eq:ratio}
\text{Ratio } \mathcal R \coloneqq \frac{1}{100} \sum_{j=1}^{100} \min\{1,\chi(U_j,\eta)\}
\end{equation}
of 100 experiments with resulting matrices $U_j$. To evaluate how well our rotation matrices $U$ obtained from noisy data sparsify the noiseless matrices, we apply the matrices $U$ to $\mathcal H(J)$. The results are displayed in Figure~\ref{fig:Truncation_random_hessians} where the ratio of the algorithm not finding the optimal sparsity is plotted over the level $\eta.$ It shows that initialization with grid search for small $h$ is beneficial especially in dimension $d\in\{4,5\}$ while the comparison of manifold optimization by means of Landing or Riemannian gradient descent is inconclusive.

\begin{figure}[ht!]
\centering
\includegraphics[width=\linewidth]{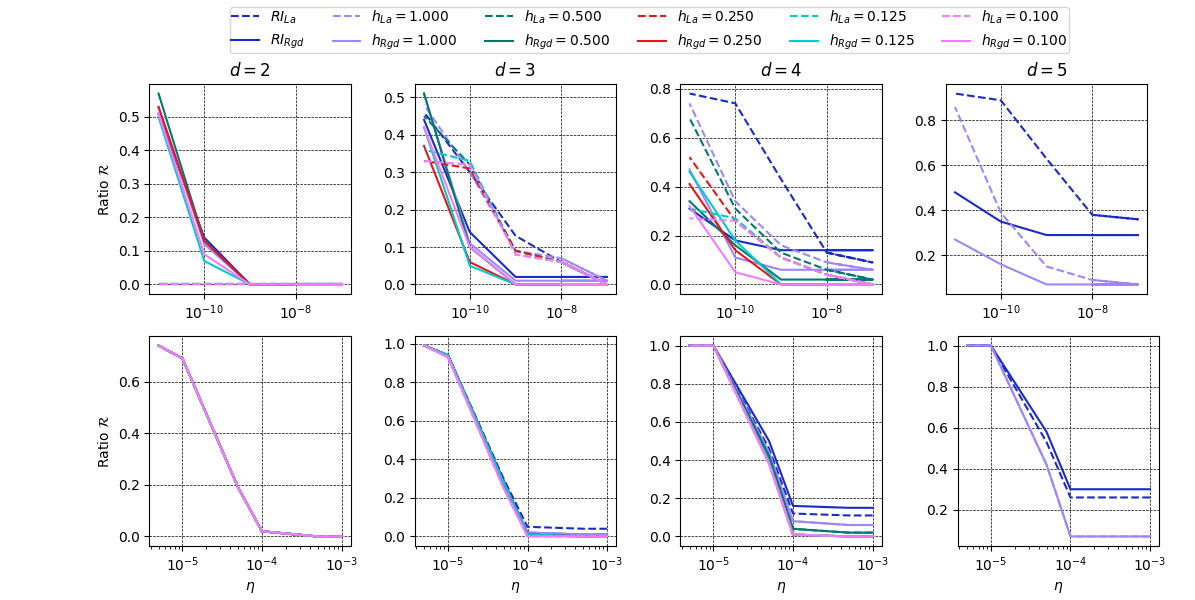}
\caption{Failure ratio $\mathcal R$ \eqref{eq:ratio} of suboptimal joint sparsity reconstructions for a given thresholding parameter $\eta$. First row: Clean data. Second row: Noisy data with additive random Gaussian noise $\mathcal{N}(0,\sigma), \sigma=10^{-3}$.}
\label{fig:Truncation_random_hessians}
\end{figure}

\subsection{Performance of the sparsifying algorithm on test functions}\label{subsec:test_functions}
To illustrate the performance of our three-step algorithm from Section \ref{schritt:1} consisting of vertex minimization, finest connected component decomposition and sparse component decomposition we apply it to 50 functions. Each test function $f:\mathbb B_r(d)\to \R$ with $10\leq d \leq 15$ arguments is determined in the following way. First, a function $\tilde f$ with sparse $\mathcal E(\tilde f)$ is constructed by randomly partitioning $[d]$ into connected components with 2 to 4 vertices. For each of the components, the edges $(j,k)$ are picked at random. Then, we set $\tilde f$ as 
\[
\tilde f(x) \coloneqq \sum_{(j,k) \in \mathcal E(\tilde f)} c_{jk} g_{jk,1}(x_j)g_{jk,2}(x_k)
\]
where the coefficients $c_{jk}$ are drawn from the interval $[5,20]$ and functions $g_{jk,1}, g_{jk,2}$ are drawn from the set
\[
S:=\set{x+t,\,  x^t,\, \sqrt[3]{x^2+t^2},\, \sin\left(tx\right), \, \cos\left(tx\right),\, e^{-(x-t)^2}: t\in\{1,2,3\}}.
\]
At last, a random rotation of the coordinates $f(x) = \tilde f(Rx)$ is applied.  
Moreover, to show that our three-step algorithm is robust towards additive noise, consider the noise function
\begin{equation}\label{eq:noise_function}
N(x) \coloneqq \frac{1}{2000}\sum_{\mu \in G_d}\exp\left(-\frac{1}{2}(x-\mu)^{T}Z^{-1}(x-\mu)\right),
\end{equation}
where $G_d=\{(x_1,\ldots,x_d):x_i\in\{-1/2,\; 3/2\}\}, Z=0.5\boldsymbol{\mathrm{I}}_{d}.$ The noisy test functions are then
\begin{equation}\label{eq:noisy_func}
f_{n} := f + N.
\end{equation}
We sample for each function $N = 100d$ points where we compute the gradient and Hessians in order to apply the algorithm.
%As in section~\ref{subsec:numeric_random_hessian} we compare the convergence and sparsifying performance. For the third step, we used either random initialization or grid search initialization followed by either landing or retraction manifold optimization for the sparse component decomposition. Figure \ref{fig:sub1} and Figure \ref{fig:functions} compare the three-step algorithm for these combinations. 
We compare the convergence of the different manifold optimization as follows. For a function $f$ resp. $f_{n}$ we apply the first two steps, the vertex minimization is done via SVD and the finest component decomposition \eqref{eq:block} which always resulted in the correct number of relevant variables resp. block sizes. Then, for each collection of blocks $\mathcal H_k = \{ H^{U,k}_n: ~ 1 \le n \le N \}$, $1 \le k \le K$, we use either random initialization or grid search initialization followed by either Landing or Riemannian gradient descent for $r$ steps on each block. Afterwards, we compute the loss \eqref{loss:h} of each block, sum over all blocks  
\begin{equation}\label{eq:l_bar_h}
\overline\ell_{\mathcal H} \coloneqq 
\sum_{k=1}^K \ell_{\mathcal H_k}.
\end{equation}
and take the mean over $50$ trials. The results shown in Figure \ref{fig:sub1} and Figure \ref{fig:functions} indicate that grid search with smaller $h$ converges faster and to lower function values than random initialization.
For comparing the sparsifying performance, we compute the ratio of functions where the algorithm does not find the optimal sparsity of Hessians. It is smaller for grid search initialization with small $h$ compared to random initialization as can be seen in the second column of Figure \ref{fig:sub1} and Figure \ref{fig:functions}. This is true for both the clean and the noisy functions. There is no substantial difference for the comparison of the Landing method and the Riemannian gradient descent. 
\begin{figure}[htbp]
\centering
\centering
\includegraphics[width=\textwidth]{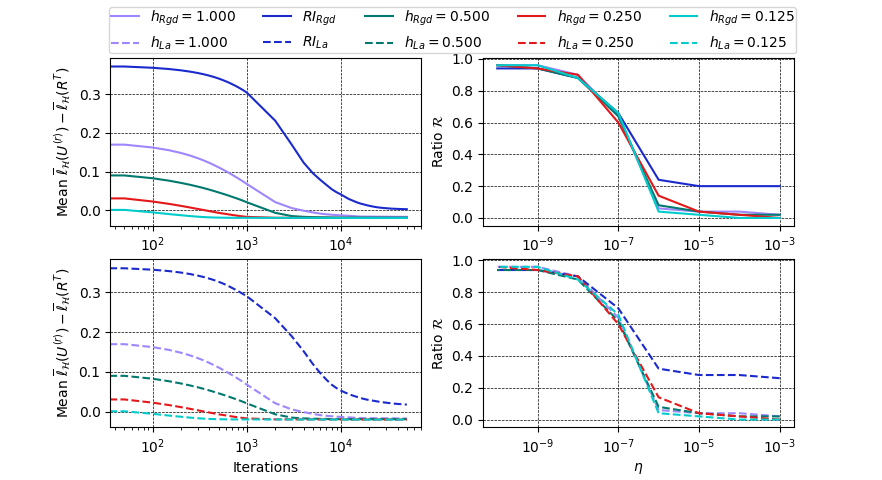}
\caption{
	%\textcolor{red}{OLEH COMMENT: This description is way too long. Also: the bottom row is just Landing and not noisy data. Noisy data is in the next figure.}
	Mean optimality gap $\overline\ell_{\mathcal H}(U^{(r)}) - \overline\ell_{\mathcal H}(R^T)$ and failure ratio $\mathcal R$ for 50 noise-free functions, see \eqref{eq:l_bar_h} and \eqref{eq:ratio} respectively.  
	% denotes the ratio of the functions where the sparsity of the Hessians, after applying the full three-step algorithm, is not optimal at level $\eta$ to the functions where it is, see also . 
	%RI\_La: Random initialization with Landing algorithm, RI\_Rgd: Random initialization with Riemannian gradient descent, h\_La=h: Grid search initialization with grid size $h$ and Landing algorithm, h\_Rgd=h: Grid search initialization with grid size $h$ and Riemannian gradient descent.
	RI denotes random initialization and $h$ uses the grid search with the corresponding grid density value. Subscripts Rgd and La stand for Riemannian gradient descent and Landing algorithm, respectively.}
\label{fig:sub1}
\end{figure}
\begin{figure}
\centering
\includegraphics[width=\linewidth]{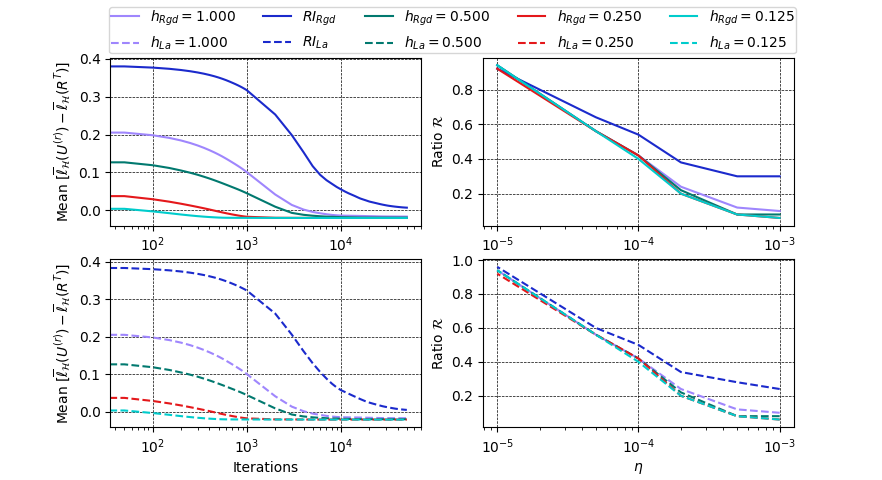}
\label{fig:prob1_6_2}
\caption{\small{Reconstructions for noisy test functions $f_n$ with notation as in Figure \ref{fig:sub1}.
		% but $\overline\ell_{\mathcal H}(U_{k, n})$ and Ratio $\mathcal R$ for the noisy functions $f_{i,n}.$
}}
\label{fig:functions}
\end{figure}
\subsection{ANOVA decomposition} 
To illustrate the dependency of the higher-order ANOVA terms on the first- and second-order derivatives discussed in Section~\ref{sect:decmark} we will consider two $7$-dimensional sparse additive test functions $f:\mathbb B_r(7) \rightarrow \R$ defined as
\begin{enumerate}
\item $\tilde{f}^1(x) =  5e^{-(x_1-1)^2}(x_4+1)+ 7\sin(2x_1)x_7^3+ 10\cos(2x_2)(x_5+3)$ where the connected components of $\mathcal{G}(f)$ correspond to
\begin{equation}
	\mathcal{B}^1 = \set{\set{x_1, x_4, x_7}, \set{x_2, x_5}}
\end{equation}
\item $\tilde{f}^2(x)= 5e^{-(x_1-1)^2}\cos(3x_4) +10x_1x_7^3 +8\sin(x_2)\cos(x_7) + 12\cos(2x_3)\sin(3x_5) + 6x_5x_6$ with connected components corresponding to
\begin{equation}
	\mathcal{B}^2 = \set{\set{x_1, x_2, x_4, x_7}, \set{x_3, x_5, x_6}}.
\end{equation}
\end{enumerate}
For each test function $\tilde{f}^i, i=1,2$ we randomly draw an orthogonal matrix $R^i \in O(7)$ such that $f^i:=\tilde{f}_{R^i}^i$ are not sparse anymore. Additionally, we add the noise function $N$ from equation~\eqref{eq:noise_function} to obtain the noisy functions $f^{i,n} = f^i + N$. As Table \ref{tab:anova_f_R} shows $f^i$ has no first- or second-order vanishing ANOVA terms anymore.\\ 
\begin{table}[htpb]
\centering
\scalebox{0.95}{\begin{tabular}{|r|c|c|c|c|}
		\hline \xrowht{7pt}
		$f=$& \multicolumn{2}{c|}{$f^1$}&\multicolumn{2}{c|}{$f^2$}\\[0.5em]
		\hline
		\multicolumn{1}{|r|}{$d_{sp}=$}&$1$&$2$&$1$&$2$\\
		\hline
		$\mathcal{A}(f, d_{sp}, \infty)$&$0.6788$&$0.2465$&$0.3421$&$0.5783$\\
		\hline
		%$\mathcal{G}(d_{sp}, 2)$&$0.4159$&$0.0646$&$0.2302$&$0.1425$\\
		%\hline
		$\mathcal{A}(f, d_{sp}, 1)$ &$0.3614$&$0.0453$&$0.2102$&$0.1203$\\
		\hline
\end{tabular}}
\caption{Minimal values of the $L^p$-norm, $p\in \{1, \infty\}$ among all first- and second-order ANOVA terms, $\mathcal{A}(f, d_{sp}, p) := \min_{\substack{\uu \subseteq [d]\\ |\uu|=d_{sp}}} \norm{f_{\uu,A}}_p$.}
\label{tab:anova_f_R}
\end{table}

We apply our three-step algorithm with the Riemannian gradient descent and a grid-search initialization procedure with step-size $h=1$ to the functions $f^i,f^{i,n}$ to obtain the orthogonal matrices $U_i, U_{i,n}.$
Table \ref{tab:anova_derivative} shows that the matrices $U_i, U_{i,n}$ both applied to $f^i$ yield the correct number of first- and second-order derivatives vanish approximately in the sense that the empirical $L_1$ resp. $L_\infty$ norm is small. Then by Proposition \ref{prop: term bound} all ANOVA terms $(f^i_{U_i})_{\mathbf{u}},(f^{i}_{U_{i,n}})_{\mathbf{u}}$ have to be approximately zero for all $\mathbf{u}\subseteq [7]$ which contains a first- or second-order term that vanishes. This is confirmed empirically in Figures \ref{fig:comp_anova_gradient_f1},\ref{fig:comp_anova_hessian_f1},\ref{fig:comp_anova_hessian_f2} where the ANOVA decomposition and its $p$-norm was computed empirically.
\begin{table}[htpb]
\centering
\scalebox{0.95}{\begin{tabular}{|r|c|c|c|c|c|c|}
		\hline \xrowht{10pt}
		$f=$&$\tilde{f}^1$&$f^1_{U_1}$&$f^{1}_{U_{1,n}}$&$\tilde{f}^2$&$f^2_{U_2}$&$f^{2}_{ U_{2,n}}$\\[0.5em]
		\hline
		$G(f, \infty)$&$2$&$2$&$2$&$0$&$0$&$0$ \\
		\cline{2-7}
		$G(f, 1)$&$2$&$2$&$2$&$0$&$0$&$0$ \\
		\hline
		$H(f, \infty)$&$18$&$18$&$17$&$16$&$15$&$14$ \\
		\cline{2-7}
		$H(f, 1)$&$18$&$18$&$18$&$16$&$15$&$15$ \\
		\hline
\end{tabular}}
\caption{Number $G(f, p):=\big|\{i \in [d]:\|\partial_{\{i\}}f\|_p \leq 10^{-4}\}\big|$ of small first-order derivatives and number $H(f, p) := \big|\{\{i,j\}\subset [d]:  i\neq j \text{ and } \|\partial_{\{i,j\}}f\|_p \leq 10^{-4}\}\big|$ of small second-order derivatives for the ground truth functions $\tilde{f}^i$ and for $f^i_{U_i}, f^{i}_{U_{i,n}}$ where $U_i, U_{i,n}$ is the output of the sparsifying algorithm.}
\label{tab:anova_derivative}
\end{table}
\begin{figure}
\centering
\begin{minipage}{.49\linewidth}
	\includegraphics[width=\linewidth]{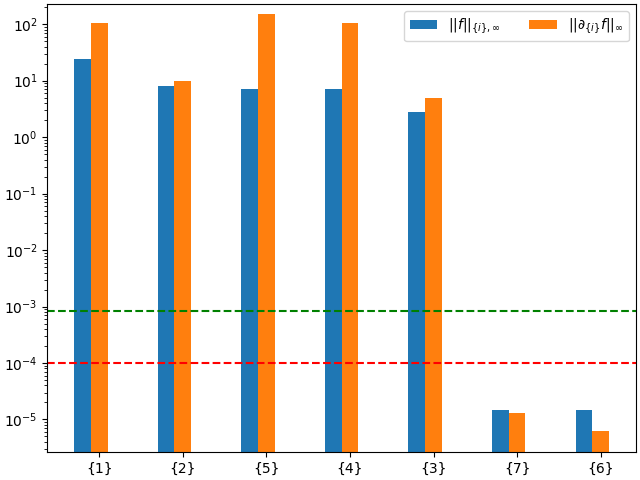}
	%\caption{$L_{\infty}$}
	%\label{fig:enter-label}
\end{minipage}
\begin{minipage}{.49\linewidth}
	\includegraphics[width=\linewidth]{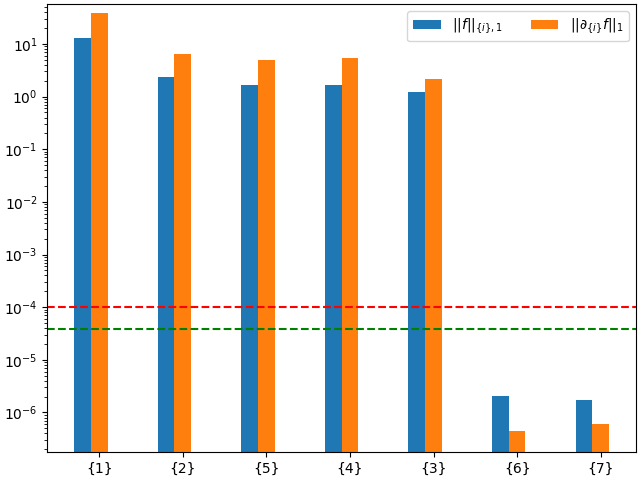}
	%\caption{$L1-norm$}
	%\label{fig:enter-label}
\end{minipage}
\begin{minipage}{.49\linewidth}
	\includegraphics[width=\linewidth]{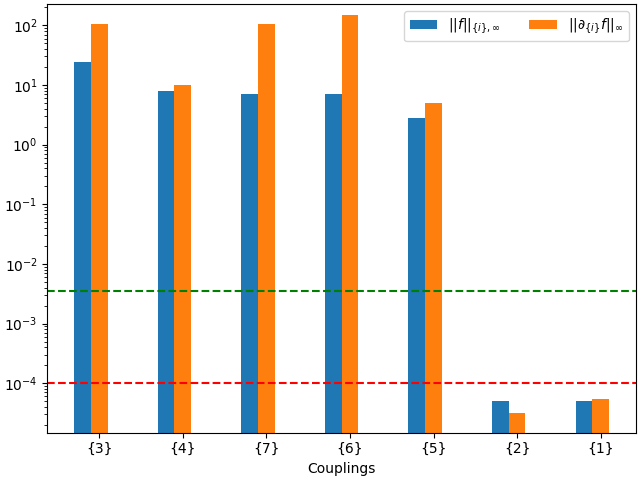}
	%\caption{$L1-norm$}
	%\label{fig:enter-label}
\end{minipage}
\begin{minipage}{.49\linewidth}
	\includegraphics[width=\linewidth]{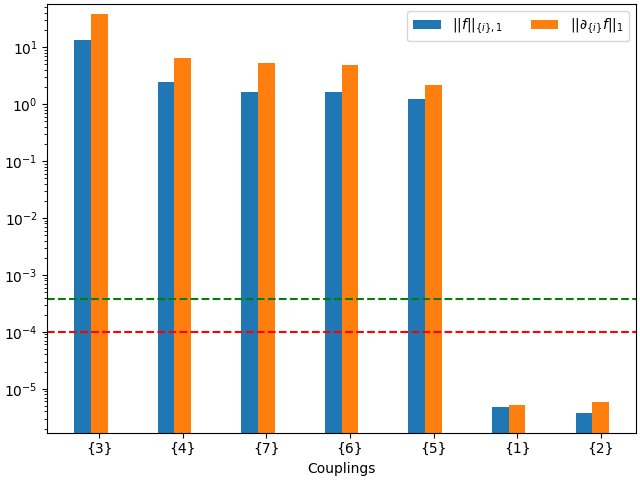}
	%\caption{$L1-norm$}
	%\label{fig:enter-label}
\end{minipage}
\caption{First row: $f =f^1_{U_1}$. Second row: $f=f^{1}_{U_{1,n}}$. First column: $p=\infty.$ Second column: $p=1.$ Blue bars:  $\|f\|_{\{i\},p}:= \max_{\{i\} \subseteq \uu \subseteq [7]}\norm{f_{\uu,A}}_p$, in decreasing order. Orange bars: $L^p$-norm of the corresponding first-order partial derivative $\|\partial_{\{i\}} f\|_p$. The green dashed line represents  $2^6\max\{\|\partial_{\{i\}} f\|_p:\|\partial_{\{i\}} f\|_p \leq 10^{-4}\}$ corresponding to Theorem \ref{prop: term bound} and the red dashed line represents the truncation value $10^{-4}$. }
\label{fig:comp_anova_gradient_f1}
\end{figure}
\begin{figure}
\centering
\begin{minipage}{.49\linewidth}
	\includegraphics[width=\linewidth]{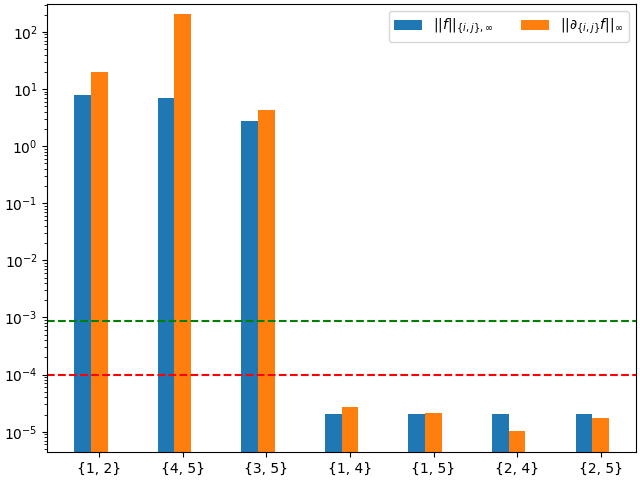}
	%\caption{$L1-norm$}
	%\label{fig:enter-label}
\end{minipage}
\begin{minipage}{.49\linewidth}
	\includegraphics[width=\linewidth]{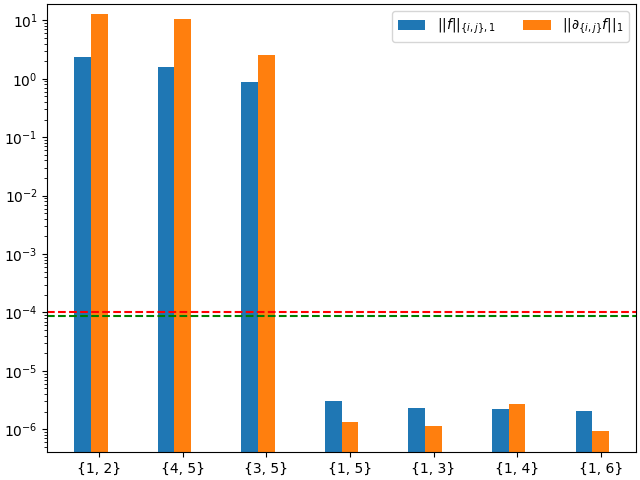}
	%\caption{$L1-norm$}
	%\label{fig:enter-label}
\end{minipage}
\begin{minipage}{.49\linewidth}
	\includegraphics[width=\linewidth]{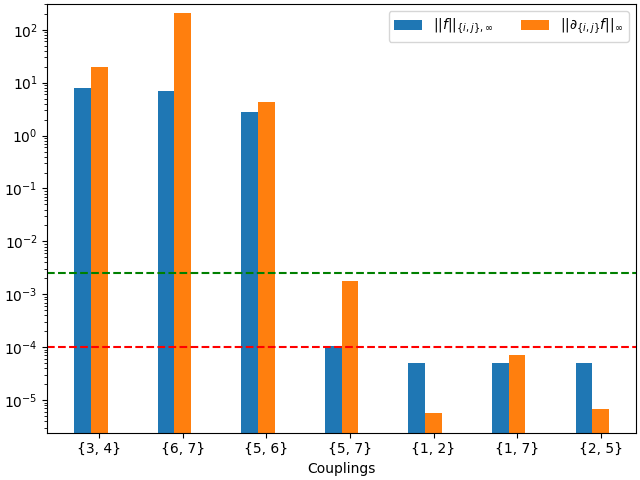}
	%\caption{$L_{\infty}$}
	%\label{fig:enter-label}
\end{minipage}
\begin{minipage}{.49\linewidth}
	\includegraphics[width=\linewidth]{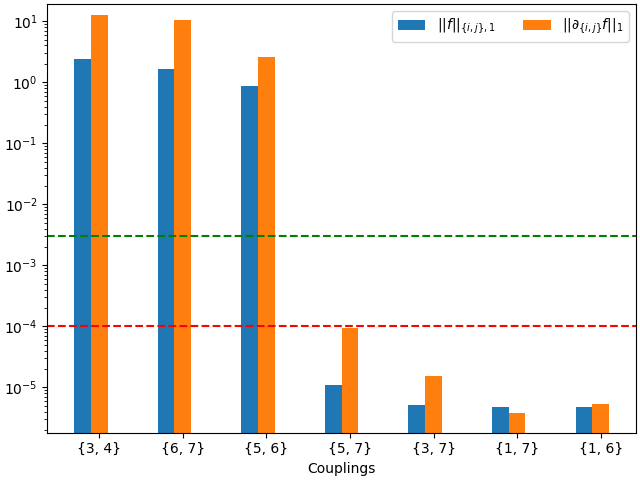}
	%\caption{$L1-norm$}
	%\label{fig:enter-label}
\end{minipage}
\caption{First row: $f =f^1_{U_1}$. Second row: $f=f^{1}_{  U_{1,n}}$. First column: $p=\infty.$ Second column: $p=1.$ Blue bars: $\|f\|_{\{i,j\},p}:= \max_{\{i,j\} \subseteq \uu \subseteq [7]}\norm{f_{\uu,A}}_p$ in decreasing order. Orange bars: corresponding $L^p$-norm of the second-order partial derivative $\|\partial_{\{i,j\}} f\|_p$. Only the largest $8$ terms are displayed. The green dashed line represents  $2^5\max\{\|\partial_{\{i,j\}} f\|_p:\|\partial_{\{i,j\}} f\|_p \leq 10^{-4}\}$ corresponding to Theorem \ref{prop: term bound} and the red dashed line represents the truncation value $10^{-4}$.}
\label{fig:comp_anova_hessian_f1}
\end{figure}

\begin{figure}
\centering
\begin{minipage}{.49\linewidth}
	\includegraphics[width=\linewidth]{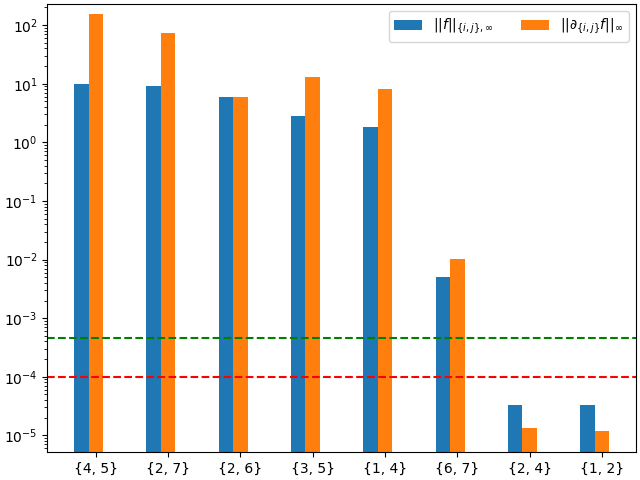}
	%\caption{$L1-norm$}
	%\label{fig:enter-label}
\end{minipage}
\begin{minipage}{.49\linewidth}
	\includegraphics[width=\linewidth]{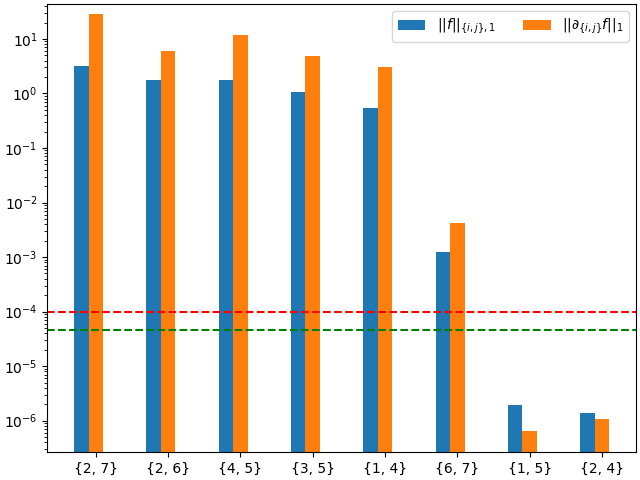}
	%\caption{$L1-norm$}
	%\label{fig:enter-label}
\end{minipage}
\begin{minipage}{.49\linewidth}
	\includegraphics[width=\linewidth]{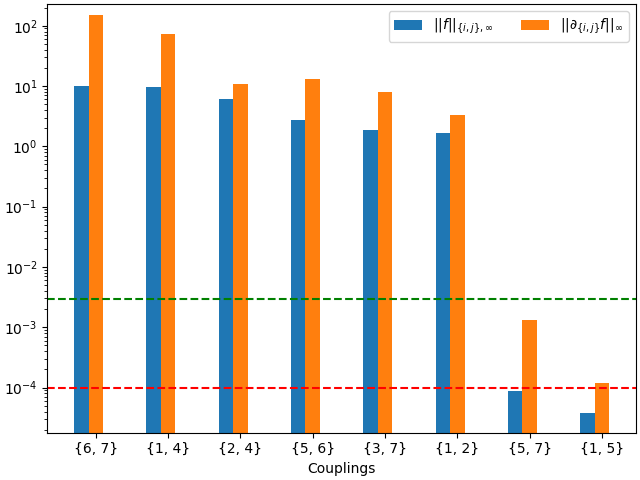}
	%\caption{$L1-norm$}
	%\label{fig:enter-label}
\end{minipage}
\begin{minipage}{.49\linewidth}
	\includegraphics[width=\linewidth]{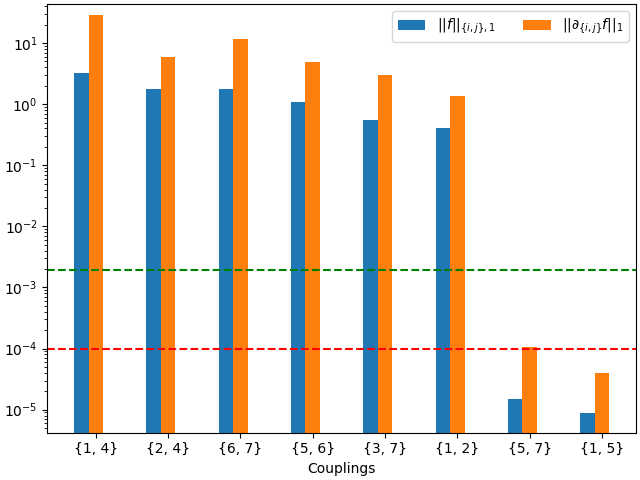}
	%\caption{$L1-norm$}
	%\label{fig:enter-label}
\end{minipage}
\caption{First row: $f =f^2_{U_2}$. Second row: $f=f^{n}_{  U_{2,n}}$. First column: $p=\infty.$ Second column: $p=1.$ Blue bars: $\|f\|_{\{i,j\},p}:= \max_{\{i,j\} \subseteq \uu \subseteq [7]}\norm{f_{\uu,A}}_p$ in decreasing order. Orange bars: corresponding $L^p$-norm of the second-order partial derivative $\|\partial_{\{i,j\}} f\|_p$. Only the largest $8$ terms are displayed. The green dashed line represents  $2^5\max\{\|\partial_{\{i,j\}} f\|_p:\|\partial_{\{i,j\}} f\|_p \leq 10^{-4}\}$ corresponding to Theorem \ref{prop: term bound} and the red dashed line represents the truncation value $10^{-4}$.} 
\label{fig:comp_anova_hessian_f2}
\end{figure}

%------------------------------------------
\section*{Acknowledgement}
F.B. acknowledges funding by the BMBF 01|S20053B project SA$\ell$E.
%and by the German Research Foundation (DFG) with\-in the project STE 571/16-1 SUPREMATIM
is gratefully acknowledged.
Ch.W. acknowledges funding by the DFG within the SFB “Tomography Across the Scales” (STE 571/19-1, project
number: 495365311). %is gratefully acknowledged. 
%---------------------------------------------------------------

%\bibliographystyle{abbrv}
%\bibliography{citations}

%%%%%%%%%%%%%%%%%%%%%%%%%%%%%%%%%%%%%%%%%%%%%%%%%%%%%%%%%%%%%%%%%%%%
\appendix
\section{Proof of Theorem \ref{prop: term bound}} \label{app:A}
%------------------------------------------------------------------
To prove the theorem, we need an auxiliary lemma.
The anchored decomposition can also be described via integrals of mixed partial derivatives as mentioned in \cite[Example 2.3]{KuoSloan2010} and for Sobolev spaces in \cite[Lemma 6]{hefter2016equivalence}. 

\begin{lemma}\label{prop:anchor_term}
Let $f:D \rightarrow \R$ such that $f \in C^{|\uu|}(D)$. Then the summands in anchored decomposition \eqref{dec_c} can be represented as
\begin{equation}\label{anchor:h}
	f_{\uu,c} = \int_{D_\uu} \partial_{t_\uu} f(t_\uu, c_{[d]\setminus \uu}) M(t_\uu, x_\uu) \, \dd \lambda_\uu (t_\uu),
\end{equation}
where
\begin{equation}
	M_{\uu}(t_{\uu}, x_{\uu}) \coloneqq \prod_{i \in \uu} M_i(t_i, x_i), \quad
	M_i(t_i, x_i) \coloneqq \begin{cases}
		1, & c_i < t_i < x_i,\\
		-1, & x_i < t_i < c_i,\\
		0 & \text{ otherwise}.
	\end{cases}
\end{equation}
\end{lemma}

\begin{proof}
We show the assertion  by induction over $|\uu|$ for $\uu \subseteq [d]$. 
Without loss of generality assume that $c_i <x_i$ for all $i\in [d].$ 
If $\uu=\emptyset$ then the assertion follows directly from the definition of the projection operator. 
Let $\emptyset \neq \vv \subset [d]$ 
and let 
$j \in [d]\setminus \vv$ be arbitrary but fixed. 
Define $ c^{t_j} \coloneqq (c_1, \cdots, c_{j-1}, t_j, c_{j+1}, \cdots, c_d).$ 
Assume that the assertion holds for $\vv$. We show that it also holds for $\uu=\vv\cup \{j\}.$ The Leibniz integral rule~\cite[Theorem~6.2]{GRIEBEL} implies on the one hand 
\begin{align}
	h_{\uu,c}(x_\uu) &= \int_{D_\uu}\partial_j \partial_\vv f(t_\vv,t_j, c_{[d]\setminus \uu})M_\vv(t_\vv, x_\vv) M_j(t_j, x_j) \, \dd\lambda_\vv(t_\vv) \, \dd \lambda_j(t_j) \\
	&= \int_{I_j}\partial_j\left(\int_{D_\vv} \partial_\vv f(t_\vv, c_{[d]\setminus \vv}^{t_j}) M_\vv(t_\vv, x_\vv)\, \dd \lambda_\vv(t_\vv)\right)M_j(t_j, x_j) \, \dd \lambda_j(t_j),
\end{align}
and on the other hand the induction step yields
\begin{align}
	h_{\uu,c}(x_\uu) 
	&=\int_{-1}^{1}\partial_j\underbrace{\left(\Bigl(\prod_{i\in \vv} \bigl(\text{Id}-P_{i,  c^{t_j}}\bigr)\Bigr)P_{[d]\setminus \vv,  c^{t_j}}f(x)\right)}_{=\tilde f(x_\vv,  c^{t_j}_{[d]\setminus \vv}) }M_j(t_j, x_j) \, \dd \lambda_j(t_j)\\
	& = 
	\int_{c_j}^{x_j}\partial_j \tilde f(x_\vv,  c^{t_j}_{[d]\setminus \vv}) \dd \lambda_j(t_j)= \int_{c_j}^{x_j}\partial_j \tilde f(x_\vv, t_j, c_{[d]\setminus \{j\}}) \, \dd \lambda_j(t_j) \\
	&= 
	\tilde f(x_\vv, x_j, c_{[d]\setminus \{j\}})- \tilde f(x_\vv, c_j, c_{[d]\setminus \{j\}})   =(\text{Id}-P_{j,c}) \tilde f(x_\vv, x_j, c_{[d]\setminus \{j\}}) \\
	&= 
	(\text{Id}-P_{j,c}) \tilde f(x_\vv, c^{x_j}_{[d]\setminus \vv}) 
	=  
	(\text{Id}-P_{j,c})\Bigl(\prod_{i\in \vv}  \bigl(\text{Id}-P_{i,c^{x_j}}\bigr)\Bigr)P_{[d]\setminus \vv, c^{x_j}}f(x).
\end{align}
Since $j \notin \vv$, we have
\begin{equation}
	\prod_{i\in \vv}  \bigl(\text{Id}-P_{i,c^{x_j}}\bigr) =  \prod_{i\in \vv}  \bigl(\text{Id}-P_{i,c}\bigr), 
\end{equation}
and the definition of the projection operator and of $c^{x_j}$ imply that
\begin{equation}
	P_{[d]\setminus \vv, c^{x_j}}f(x) 
	= f(x_\vv, c^{x_j}_{[d]\setminus \vv})
	=f(x_\vv, x_j, c_{[d]\setminus (\vv\cup\{j\})}) 
	= P_{[d]\setminus \uu, c}f(x)
\end{equation}
which concludes the proof.
\end{proof}

\begin{proof}[Proof of Theorem \ref{prop: term bound}.]
i) By Lemma \ref{prop:anchor_term}, we have for sufficiently smooth $f$ that
\begin{align}
	\|f_{\vv,c}\|_\infty& = \left\|\int_{D_\vv} \partial_{t_\vv} f(t_\vv, c_{[d]\setminus \vv})M(t_\vv, x_\vv) \, \dd\lambda_\vv(t_\vv)\right\|
	\leq \|\partial_{t_\vv} f\|_{\infty} \lambda_\vv(D_\vv).
\end{align}
For $\mathbf{w}\subseteq [d]$ let $f^{x_{\mathbf{w}}}:\R^{d-|\mathbf{w}|}\to \R$ 
be the family of functions defined by $f^{x_{\mathbf{w}}}(x_{[d]\setminus \mathbf{w}}):=f(x_{\mathbf{w}},x_{[d]\setminus \mathbf{w}})$. For $\mathbf{w}:=\uu \setminus \vv$ we have that
\begin{align}\label{eq:neuf}
	f_{\uu,c}&= 
	\prod_{i\in\uu}(\text{Id}-P_{i,c})P_{[d]\setminus \uu,c}f 
	= \prod_{i\in\mathbf{w}}(\text{Id}-P_{i,c})\prod_{j\in\vv}(\text{Id}-P_{j,c})P_{([d]\setminus \mathbf{w})\setminus\vv}f\\
	&= \prod_{i\in\mathbf{w}}(\text{Id}-P_{i,c})f^{x_{\mathbf{w}}}_{\vv}(x_\vv).
\end{align}
Combining with $\left|(\text{Id}-P_{i,c})g(x_{[d]\setminus i})\right|\leq 2\|g\|_\infty$ for functions $g$, we obtain
\begin{align}
	\|f_{\uu,c}\|_{\infty} &=  \left\|\prod_{i\in\mathbf{w}}(\text{Id}-P_{i,c})f^{x_{\mathbf{w}}}_{\vv}(x_\vv)\right\|_\infty\leq 2^{|\mathbf{w}|}\sup_{x_\mathbf{w}\in D_\mathbf{w}}\{\|f^{x_{\mathbf{w}}}_{\vv}\|_\infty\}\\
	&\leq 2^{|\uu|-|\vv|}\|\partial_{x_{\vv}}f\|_\infty\lambda_\vv(D_\vv).
\end{align}

ii) This follows from part i) and the fact that we have by by Proposition \ref{anova:anchored} for the  ANOVA summands
\begin{align}
	f_{\uu,A} = \int_{D} f_{\uu, c} \, \dd\lambda(c).
\end{align}

iii) %\textcolor{red}{OLEH COMMENT: Something is missing here:} 
For $\mathbf{w} \subseteq [d]$, 
let $f^{x_{\mathbf{w}}}$ be defined as above and for any disjoint subsets $\mathbf{s}, \mathbf{t} \subseteq [d]$ let $f^{x_{\mathbf{s}}, c_{\mathbf{t}}}(x_{[d]\setminus (\mathbf{s}\cup \mathbf{t})}) = f(x_{\mathbf{s}}, c_{\mathbf{t}}, x_{[d]\setminus (\mathbf{s} \cup \mathbf{t})})$. For $\mathbf{w} \subseteq [d]$ and any commuting linear operators $a_j, b_j, j\in \mathbf{w}$ the following holds
\begin{equation}\label{eq:commop}
	\prod_{j \in \mathbf{w}} (a_j - b_j) = \sum_{\mathbf{s} \subseteq \mathbf{w}} \left(\prod_{j \in \mathbf{s}} a_j \right) \left(\prod_{j \in \mathbf{w} \setminus \mathbf{s}} b_j \right).
\end{equation}
If we define $\mathbf{w}:=\uu \setminus \vv$ then equations~\eqref{eq:neuf}, \eqref{eq:commop}
and \eqref{proj:anchored} imply that
\begin{align}
	f_{\uu,c}= \prod_{i\in\mathbf{w}}(\text{Id}-P_{i,c})f^{x_{\mathbf{w}}}_{\vv, c_{[d]\setminus \mathbf{w}}} = \sum_{\mathbf{s} \subseteq \mathbf{w}}(-1)^{|\mathbf{w}|-|\mathbf{s}|} P_{\mathbf{w}\setminus \mathbf{s}} f_{\vv,c_{[d]\setminus \mathbf{w}}}^{x_{\mathbf{w}}}= \sum_{\mathbf{s} \subseteq \mathbf{w}}(-1)^{|\mathbf{w}|-|\mathbf{s}|} f_{\vv,c_{[d]\setminus \mathbf{w}}}^{x_{\mathbf{s}}, c_{\mathbf{w}\setminus \mathbf{s}}}.
\end{align}
By proposition~\ref{anova:anchored} the ANOVA term fulfills
\begin{align}
	\norm{f_{\uu,A}}_1 &\leq\int_{D_\uu}\int_{D}\abs{f_{\uu, c}}d\lambda(c) dx_{\uu} = \int_{D}\int_{D_{\uu}}\abs{\sum_{\mathbf{s} \subseteq \mathbf{w}}(-1)^{|\mathbf{w}|-|\mathbf{s}|} f_{\vv,c_{[d]\setminus \mathbf{w}}}^{x_{\mathbf{s}}, c_{\mathbf{w}\setminus \mathbf{s}}}} d\lambda_{\uu}(x_{\uu})d\lambda(c)\\
	&\leq \sum_{\mathbf{s} \subseteq \mathbf{w}}\lambda(D_{\mathbf{w}})\int_{D_{[d]\setminus \mathbf{s}}} \int_{D_{\vv \cup \mathbf{s}}} \abs{ f_{\vv,c_{[d]\setminus \mathbf{w}}}^{x_{\mathbf{s}}, c_{\mathbf{w}\setminus \mathbf{s}}}}d\lambda_{\vv \cup \mathbf{s}}(x_{\vv \cup \mathbf{s}}) d\lambda_{[d]\setminus \mathbf{s}}(c_{[d]\setminus \mathbf{s}})\\
	&= \lambda(D_{\mathbf{w}})\sum_{\mathbf{s} \subseteq \mathbf{w}}\int_{D_{[d]\setminus \mathbf{w}}}\int_{D_{\vv}}\int_{D_{\mathbf{w}}}\abs{f_{\vv, c_{[d]\setminus \mathbf{w}}}^{x_{\mathbf{w}}}} d\lambda_{\mathbf{w}}(x_{\mathbf{w}})d\lambda_{\vv}(x_{\vv})d\lambda_{[d]\setminus \mathbf{w}}(c_{[d]\setminus \mathbf{w}})\\
	&= 2^{|\mathbf{w}|} \lambda(D_{\mathbf{w}})\underbrace{\int_{D_{[d]\setminus \mathbf{w}}}\int_{D_{\vv}}\int_{D_{\mathbf{w}}}\abs{f_{\vv, c_{[d]\setminus \mathbf{w}}}^{x_{\mathbf{w}}} }d\lambda_{\mathbf{w}}(x_{\mathbf{w}})d\lambda_{\vv}(x_{\vv})d\lambda_{[d]\setminus \mathbf{w}}(c_{[d]\setminus \mathbf{w}})}_{=:I}.
\end{align}
Finally, the triangle inequality yields 
\begin{align}
	I & = \int_{D_{[d]\setminus \mathbf{w}}}\int_{D_{\uu}}\abs{\int_{D_\vv} \partial_{t_\vv} f(t_\vv, c_{([d]\setminus \mathbf{w})\setminus \vv}, x_{\mathbf{w}})M(t_\vv, x_\vv) d\lambda_\vv(t_\vv)}d\lambda(x_{\uu})d\lambda(c_{[d]\setminus\mathbf{w}})\\
	&\leq \int_{D_{[d]\setminus \mathbf{w}}}\int_{D_{\uu}}\int_{D_\vv} \abs{\partial_{t_\vv} f(t_\vv, c_{([d]\setminus \mathbf{w})\setminus \vv}, x_{\mathbf{w}})} d\lambda_\vv(t_\vv)d\lambda(x_{\uu})d\lambda(c_{[d]\setminus\mathbf{w}})\\
	&= \lambda_{\vv}(D_{\vv})\int_{D_{[d]\setminus \mathbf{w}}}\int_{D_{\mathbf{w}}}\int_{D_\vv} \abs{\partial_{t_\vv} f(t_\vv, c_{([d]\setminus \mathbf{w})\setminus \vv}, x_{\mathbf{w}})}d\lambda_\vv(t_\vv)d\lambda(x_{\mathbf{w}})d\lambda(c_{[d]\setminus\mathbf{w}})\\
	&= \lambda_{\vv}^2(D_{\vv})\int_{D_{[d]\setminus \uu}}\int_{D_{\mathbf{w}}}\int_{D_\vv} \abs{\partial_{t_\vv} f(t_\vv, c_{[d]\setminus \uu}, x_{\mathbf{w}})}d\lambda_\vv(t_\vv)d\lambda(x_{\mathbf{w}})d\lambda(c_{[d]\setminus\uu})\\
	&= \|\partial_{t_\vv} f\|_{1}\lambda_\vv^2(D_\vv).
\end{align}
By combining all together, we obtain
\begin{align}
	\norm{f_{\uu,A}}_1  \leq 2^{|\uu|-|\vv|}\|\partial_{x_{\vv}}f\|_1\lambda_\uu(D_\uu)\lambda_\vv(D_\vv).
\end{align}
\end{proof}

%%%%%%%%%%%%%%%%%%%%%%%%%%%%%%%%%%%%%%%%%%%%%%%%%%%%%%%%%%%%%%%%%%%%%%%%%%%%%%%%%
\section{Proof of Theorem \ref{uniqueness:block}}\label{appendix:block}

The proof requires a lot of preliminary results.
In particular, the uniqueness up to permutation of blocks and conjugation by orthogonal matrices 
relies on the theory of matrix $*$ algebras. 

A \textit{matrix $*$ algebra} $\mathcal{A}$ is a subalgebra of the $\R$ algebra of $n\times n$ matrices $M_n(\R)$ which is closed under taking the transpose i.e.  $A\in\mathcal{A}$ implies $A^T\in\mathcal{A}$. Let $A_i\in M_{n_i}(\R)$ and denote by $\diag(A_1,\ldots,A_n)$ the block diagonal matrix
\[
\begin{pmatrix}
A_1 &0&0&0\\
0   &A_2&0&0\\
\hdots&\hdots&\ddots&0\\
0&0&0&A_n
\end{pmatrix}.
\]
a subset $\mathcal{A}\subseteq\{\diag(A_1,\ldots,A_k):A_i\in M_{\tilde{n}_i}\}$ we denote by $\pi_j$ the projection onto the $j$-th block i.e. $\pi_j(\diag(A_1,\ldots,A_k))=A_j$. If $\mathcal{A}$ is a matrix $*$ algebra the map $\pi_j:\mathcal{A}\to M_{\tilde{n}}(\R)$ is a \textit{morphism of matrix $*$ algebras} which means that $\pi_j$ is a $\R$ algebra morphism fullfiling $\pi_j(A^T)=\pi_j(A)^T$.\\
Recall that an $\R$ algebra is \textit{simple} if it does not contain any proper ideals. A matrix algebra $\mathcal{A}\subseteq M_n(\R)$ has a canonical representation $\rho:\mathcal{A}\to End(\R^n)$ given by $\rho(x)=x$ which is called the \textit{regular representation} of $\mathcal{A}$. We call $\mathcal{A}$ \textit{irreducible} if the regular representation is irreducible i.e. the regular representation contains no non-trivial subrepresentation.
\begin{remark}\label{art:wedd} An important decomposition of a semisimple algebra is the Artin-Wedderburn decomposition and we will describe its close relationship to the block diagonalization. While subalgebras of $M_n(\R)$ are not semisimple in general, matrix $*$ algebras subalgebras of $M_n(\R)$ are semisimple $\R$ algebras \cite[Lemma A.3]{murota2010numerical}. Thus by the Artin-Wedderburn Theorem \cite[Theorem 5.2.4]{cohn2012basic} for every matrix $*$-subalgebra $S$ of $M_n(\R)$ there exists division algebras $D_i$ over $\R$ and $n_i\in\mathbb N$ such that 
\[
S\simeq M_{n_1}(D_1)\times\cdots\times M_{n_j}(D_j)
\]
where up to permutation of factors the $n_i$ are unique and the $D_i$ are unique up to isomorphism. Note that the individual factors $M_{n_i}(D_i)$ are simple $\R$-algebras and simple $S$ modules.\end{remark}
\begin{remark}\label{rem:generated_algebra}
Note that $U^T H_iU, 1\leq i\leq n$ has a common block diagonal form if and only if this is true for every element of the matrix $*$ algebra generated by $U^TH_iU,1\leq i\leq n$. This follows from the fact that for $U\in O(d)$ we have that 
\begin{align}
	U^T (A+B)U&=U^TAU+U^TBU,&U^TABU=U^TAUU^TBU\\
	(U^TAU)^T&=U^TA^TU.&
\end{align}
In particular if $H_1,\ldots,H_n$ generate $V_{\nabla^2f}$  and $U\in O(d)$ corresponds to a finest block diagonalization of $H_1,\ldots, H_n$ then $f_U\in \mathcal{MK}(\mathcal{GS}(f))$. Furthermore up to a permutation matrix every element of $\mathcal{MK}(\mathcal{GS}(f))$ is of this form. Thus we can use block diagonalization results for matrix $*$ algebras in order to find a suitable $U\in O(d).$
\end{remark}
While there is algorithms about block diagonalization of matrix $*$ algebras to our knowledge the uniqueness of the blocks up to orthogonal conjugation of the blocks is not explicitly stated in the corresponding papers which is why we will briefly discuss it in the following.\\
The general strategy for proving the uniqueness of a finest block diagonalization up to conjugation of the blocks by orthogonal matrices is as follows. Let $H_1,\ldots,H_n\in M_d(\R)$ and let $U_1, U_2\in O(d)$ correspond to finest block diagonalizations. Denote by $B_i$ the matrix $*$-algebra generated by $U_i^TH_1U_i,\ldots,U_i^TH_nU_i$. 
Then
\begin{itemize}
\item Show that there exists simple matrix $*$-subalgebras $\mathcal{T}_{i,j}\subseteq B_i$ such that \[B_i=\{\diag(T_{i,1},\ldots,T_{i,n_i}):T_{i,j}\in \mathcal{T}_{i,j}\}.\] 
Show that $n_1=n_2$.
\item Denote by $\varphi:B_1\to B_2$ the matrix $*$-isomorphism $\varphi: x\mapsto U_2^TU_1xU_1^TU_2$ and by $\varphi\big|_{\mathcal{T}_{1,j}}:\mathcal{T}_{1,j}\to B_2$ the restriction to $\mathcal{T}_{1,j}$. Show up to permutation of the indices $j$ it holds $\varphi\big|_{\mathcal{T}_{1,j}}\left(\mathcal{T}_{1,j}\right)= \mathcal{T}_{2,j}$. Consequently \[\varphi\big|_{\mathcal{T}_{1,j}}:\mathcal{T}_{1,j}\to \mathcal{T}_{2,j}\] is a matrix $*$-isomorphism. 
\item To conclude the argument we show that a matrix $*$-algebra isomorphism between simple matrix $*$-algebras $S_1,S_2\in M_n(\R^d)$ which are of finest decomposition can be described as conjugation of the blocks by orthogonal matrices.\\
To be more precise let $S_1\subseteq \{\diag(A_1,\ldots, A_n):A_j\in M_{m}(\R)\}\supseteq S_2$ be simple matrix $*$-subalgebras such that the image of the projection $\pi_j$ onto the $j$-th block is irreducible. Then for every matrix $*$-isomorphism $\phi:S_1\to S_2$ there exists $V_j\in O(d)$ such that \[\phi(x)=\diag(V_1,\ldots,V_n)^Tx\diag(V_1,\ldots,V_n).\]
\end{itemize}
Overall we see that up to a permutation of blocks finest block diagonalizations are unique up to a orthogonal conjugation of the individual blocks.\\
The following theorem guarantees the existence of a finest block diagonalization.

\begin{theorem}{\cite[Theorem 1]{murota2010numerical}}\label{block:murota}
Let $\mathcal{T}$ be a $*$ subalgebra of $M_n(\R)$. Then there exists $U\in O(n)$ and simple subalgebras $\mathcal{T}_j\subseteq M_{n_j}$ s.t.
\begin{align}
	U^T \mathcal{T} U = \{\mathrm{diag}(T_1,\ldots,T_k):T_i\in\mathcal{T}_i\}.
\end{align} 
Furthermore for a simple subalgebra $\mathcal{T}\subseteq M_n(\R)$ there exists an irreducible subalgebra $\mathcal{T}^i\subseteq M_{\hat{n}}(\R)$ and a $V\in O(n)$ s.t.
\[
V^T \mathcal{T} V=\{\mathrm{diag}(A,\ldots,A):A\in\mathcal{T}^i\}.
\]

\end{theorem}

\begin{remark}
Theorem \ref{block:murota} also illustrates that for two isomorphic matrix $*$-algebras $S_1\simeq S_2$ which are both contained in $M_n(\R)$ it is not necessarily true that they are isomorphic as vector spaces. Take e.g. the irreducible matrix $*$-algebra $S=M_m(\R)$ and $S_1 = \{\diag(A,0): A\in S\}\subset M_{2m}(\R)$ and $S_2=\{\diag(A,A):A\in S\}\subset M_{2m}(\R)$. 
\end{remark}
\begin{remark}\label{definitions:equiv}
All finest block diagonalizations correspond to a matrix $*$-algebra decomposition as in Theorem \ref{block:murota} i.e. the block of a finest block diagonalization are irreducible matrix $*$-algebras. Otherwise they could be further refined. By Lemma \ref{unique:simple} these irreducible blocks can be grouped into simple algebras algebras.  
\end{remark}
We start the rigorous proof with the following auxiliary lemma.
\begin{lemma}\label{onto:irr}
Let $U$ correspond to a finest block diagonalization of $H_1,\ldots, H_n$ with block sizes $(d_1,\ldots,d_k)$. Let $A$ be the matrix $*$-subalgebra of $M_d(\R)$ generated by $H_1,\ldots,H_n$. Then also $U^T AU$ is a matrix $*$-algebra and every $a\in A$ is blockdiagonal with block sizes $(d_1,\ldots,d_k)$. Furthermore the projection $\pi_i:U^TAU\to M_{d_i}(\R)$ is a $*$-algebra morphism onto an irreducible subalgebra of $M_{d_i}(\R)$.
\end{lemma}
\begin{proof}
It is immediate that $U^TAU$ is a matrix $*$-algebra which is block diagonal and that $\pi_i$ is a $*$-algebra morphism, see also Remark \ref{rem:generated_algebra}. Assume that the image of $\pi_i$ was not irreducible. Then according to Theorem \ref{block:murota} this block could be further refined by an orthogonal conjugation which would contradict that $U$ corresponds to a finest block diagonalization.
\end{proof}

Next we make sure that the simple components of the Artin-Wedderburn decomposition show up as, not necessarily finest, blocks in the finest block decomposition.
\begin{lemma}\label{unique:simple}
Let $H_1,\ldots,H_n\in \R^{d\times d}$ and assume that $U\in O(d)$ corresponds to a finest block diagonalization. Let $A$ be the matrix $*$-algebra generated by $H_1,\ldots,H_n$ and let $d_1,\ldots,d_k$ be the size of the blocks. Then there exists unique (up to block permutation) simple matrix  $*$-algebras $\mathcal{T}_i\subseteq U^T A U$ such that
\begin{align}
	U^TAU=\{\diag(T_1,\ldots,T_k): T_i\in \mathcal{T}_i\}
\end{align}
\end{lemma}
\begin{proof}
Let $B=U^TAU$ and let $B\simeq \oplus_{i=1}^k \mathcal R_i$ be the Artin-Wedderburn decomposition into simple $B$ modules which are also subalgebras (without 1), which exists by Remark \ref{art:wedd}. By abuse of notation denote by $\mathcal R_i$ also the corresponding subalgebra of  $B$. Let $\pi_j$ be the projection of $B$ onto the $j$-th block. By Corollary \ref{onto:irr} the image of $\pi_j$ is irreducible and hence in particular simple. Thus Schur's lemma implies that $\pi_j\Big|_{\mathcal R_i}$ is either zero or an isomorphism onto the block. Since $\mathcal R_i$ are orthogonal, we have that if for indices $n,m$ we have that both $\pi_j\big|_{\mathcal R_n}\neq0$ and $\pi_j\big|_{\mathcal R_m}\neq0$, then $n=m$ and the claim follows.
\end{proof}

\begin{lemma}\label{simple:same}
Let $U_1,U_2$ correspond to finest block diagonalizations and denote by $B_i$ the matrix $*$-algebras $U_i^TAU_i$. Denote by $\mathcal{T}_{1,1},\ldots,\mathcal{T}_{1,k_1}$ resp. $\mathcal{T}_{2,1},\ldots,\mathcal{T}_{2,k_2}$ be the simple matrix $*$ algebras from Lemma \ref{unique:simple}. Then $k_1=k_2=:k$ and for $\varphi:B_1\to B_2$ defined by $\varphi(x)=U_2^TU_1x U_1^TU_2$ there exists a permutation $\sigma\in \Sigma(k)$ such that $\varphi(\mathcal{T}_{1,j})=\mathcal{T}_{2,\sigma(j)}$ and $\varphi:\mathcal{T}_{1,j}\to \mathcal{T}_{\sigma(j),2}$ is an isomorphism of matrix $*$ algebras.
\end{lemma}
\begin{proof}
By Artin-Wedderburn $k_1=k_2$. By Schur's lemma we have that $\pi_{T_{2,m}}\circ\varphi\Big|_{\mathcal{T}_{1,n}}$ is either an isomorphism or zero. Since the $T_{1,l}$ are orthogonal we have that if for $m,n$ it holds that both $\pi_{T_{2,j}}\circ\varphi\Big|_{\mathcal{T}_{1,m}}\neq0$ and $\wedge \pi_{T_{2,j}}\circ\varphi\Big|_{\mathcal{T}_{1,n}}\neq 0$ then $m=n$. Hence there exists a partition $[k_2]=\cup_{j=1}^{k_1} I_j$ with $I_j=\{n\in[k_2]: \pi_{T_{2,n}}\circ\varphi\Big|_{\mathcal{T}_{1,j}}\neq 0\}$. Since $\varphi$ is an isomporhism we know that $|I_j|\geq 1$. Using that $k_1=k_2$ we can conclude that $|I_j|=1$ and thus the claim.
\end{proof}

\begin{proposition}\label{irr:v}
Let $\mathcal{S}\subseteq M_n(\R)$ be an irreducible matrix $*$ subalgebra and let $\mathcal{T}\subseteq M_{\tilde{n}}$ be a simple matrix $*$ subalgebra isomporhic to $\mathcal S$ via a matrix $*$ morphism $\varphi:\mathcal S\to \mathcal T$. Assume furthermore that $\mathcal{T}\subseteq \{\diag(A_1,\ldots,A_k):A_i\in M_n(\R)\}$. Then there exists $V_1,\ldots,V_k\in O(n)$ such that for $V=\diag(V_1,\ldots,V_k)$ we have that $\varphi(x)=V^T\diag(x,\ldots,x)V$.
\end{proposition}
\begin{proof}
By Skolem-Noether there exists matrices $W_1,\ldots,W_k\in GL_n(\R)$ such that for $W=\diag(W_1,\ldots,W_k)$ we have that $\varphi(x)=W^{-1}\diag(x,\ldots,x)W$. This implies that for $\pi_i$ the projection of the $k$-th block we have that $\pi_i\circ \varphi$ is an isomorphism of matrix $*$ representations via the map $W_i:\R^n\to \R^n$. Thus by \cite[Lemma A.4]{murota2010numerical} there exists $V_i\in O(n)$ such that $\pi_i\circ \varphi(x)=V^TxV$ and thus the claim.
\end{proof}

\begin{corollary}\label{simple:irred}
Let $\mathcal{T}_1\subseteq M_n(\R)$ and $\mathcal{T}_2\subseteq M_m(\R)$ be simple isomorphic matrix $*$ algebra with isomorphism $\varphi$. Assume furthermore that $\mathcal{T}_i\subseteq \{\diag(A_1,\ldots,A_{k_i}):A_j\in M_{\tilde{n}}(\R)\}$ such that the image of the projection onto the $\ell$-th block is an irreducible representation. Then $\varphi$ induces an isomorphism $\varphi_{i,j}:\pi_i (\mathcal{T}_1)\to \pi_j(\mathcal{T}_2)$ and there exists a $V\in O(\tilde{n})$ such that $\varphi_{i,j}(x)=V^TxV$.
\end{corollary}
\begin{proof}
There exists an irreducible matrix $*$ algebra $\mathcal{S}$ and an isomorphism $\tau:\mathcal{S}\to \mathcal{T}_1$. By Proposition \ref{irr:v} we have that there exists $V_1,V_2$
such that $\pi_i\circ \tau(x)=V_1^Tx V_1$ and $\pi_j\circ \varphi\circ\tau(x)=V_2^TxV_2$. Thus we can use  $V:=V_1^{-1}V_2$.
\end{proof}

\begin{proposition}\label{decomp:finally}
Let $U_1,U_2$ correspond to finest block diagonalizations and denote by $B_1,B_2$ the matrix $*$ algebras $U_1^TAU_1$ resp. $U_2^TAU_2$. Let $d_{1,j},\ldots,d_{k_j,j}$ be the blocks sizes of $B_j$ and let $\varphi:B_1\to B_2$ be defined by $\varphi(x)= U_2^TU_1xU_1^TU_2$. Then $k_1=k_2=:k$ and there exists a permutation $\sigma\in\Sigma(k)$ such that $d_{i,1}=d_{\sigma(i),2}$ and a block permutation matrix $P_{\sigma}$ and $V_1,\ldots,V_k, V_i\in O(d_{i,1})$ such that $\varphi(x)=V^TP_{\sigma}^TxP_{\sigma}V$.
\end{proposition}
\begin{proof}
By Lemma \ref{unique:simple} we know that there exists a decomposition of $B_1,B_2$ into simple algebras which have disjoint block form. By Lemma \ref{simple:same} there exists a permutation of simple blocks such that $\varphi$ maps a simple block isomorphically to exactly one simple block. Since $\varphi$ does preserve the rank of matrices the corresponding simple blocks must be of same dimension. By Corollary \ref{simple:irred} the claim follows.
\end{proof}

%----------------------------------------------------------------------
\section{Algorithmic Details}\label{sec43}

\subsection{Grid search and initialization close to a global minimum}\label{subsec:grid_search}

According to Lemma \ref{lem:geo_convex}, the loss function $\ell_\varepsilon$ is not geodesically convex, and, thus, convergence to a global solution is not guaranteed. By Theorem \ref{thm: manifold optimization local}, the Riemannian gradient descent method converges locally with suitable step sizes. To find a good starting point $U\in SO(d)$, we establish a grid search procedure on $SO(d)$. Its main idea is to compute the loss on the discrete subset $\Gamma \subset SO(d)$ and initialize the gradient-descent method with the minimizer of
\begin{equation}\label{eq:P_grid_search}
U = \argmin_{V \in \Gamma} \ell_\varepsilon(V).
\end{equation}
% In view of Proposition \ref{pro:restrict_SOd} iii), we construct a suitable $\Gamma$ as a subset of $SO(d)$. 
The benefit of working with $SO(d)$ is that it can be parametrized by $d(d-1)/2$ angles. Let us consider the Jacobi rotation matrix with parameters $r \in [d-1]$ and $\alpha \in [0,2\pi)$ given by
\begin{equation}
\begin{aligned} \label{eq: rotation def}
	R_{r,r}(r,\alpha)= \cos(\alpha) &\quad  R_{r, r+1}(r,\theta) = -\sin(\alpha) & \quad R_{j,j}(r,\alpha) =1, \ j \notin \{r,r+1\}, \\
	R_{r+1, r}(r,\alpha)=\sin(\alpha), & \quad R_{r+1, r+1}(r,\alpha) = \cos(\alpha), &  R_{i,j}(r,\alpha) =0 \ \text{otherwise.} 
	% R_{j,j}(r,\alpha)=1, \ \text{if} j \notin \{r,r+1\} & \text{and} & R_{i,j}(r,\alpha)=0 \ \text{otherwise}.
\end{aligned} 
\end{equation}
Then, every $SO(d)$ matrix factorizes into a product of $d(d-1)/2$ Jacobi matrices. 
\begin{proposition}[{\cite{Hurwitz1897}}]\label{prop:recursiv_def_SOd}
Let $d \geq 2$. For every $U \in SO(d)$, there exist vectors
\[
\alpha^r = (\alpha^r_{1}, \cdots, \alpha^r_{r}) \in [0, 2\pi) \times [0, \pi)^{r-1}, \quad r \in [d-1],
\]
such that
\begin{equation}\label{eq: rotation parametrization}
	U = U(\alpha) = \prod_{r=1}^{d-1} \prod_{j=1}^{r} R(d - 1 - r + j, \alpha^r_{j}).
\end{equation}
\end{proposition}

Proposition \ref{prop:recursiv_def_SOd} provides a  parametrization of $U \in SO(d)$ by a vector of angles $\alpha \in [0, 2 \pi)^{d-1} \times [0, \pi)^{(d-1)(d-2)/2} \subset \R^{d(d-1)/2}$. 
% where
% \begin{equation}
% p :=\sum_{r=1}^{d-1} r = \frac{d(d-1)}{2}.
% \end{equation}
In the context of the grid search, the main benefit of the angle representation of $SO(d)$ is the possibility to construct lattices efficiently. Let
\[
\Theta(a,h) := \{ k h :  k \in \lceil a / h \rceil \}
\quad \text{and} \quad
\Theta(h) := \Theta(2\pi,h)^{d-1} \times \Theta(\pi, h)^{(d-1)(d-2)/2}
\]
be one- and multidimensional lattices, respectively. Then, we define $\Gamma$ in \eqref{eq:P_grid_search} as
\begin{equation}\label{eq: grid set}
\Gamma := \Gamma(h) = \{ U(\alpha) \in SO(d) : \alpha \in \Theta(h)\}.
\end{equation}
Then, for such $\Gamma$ we have the following results. 
\begin{proposition}\label{prop: grid search bound}
Let $h\in (0,1)$. Then, for every $V \in SO(d)$ there exists $\alpha \in \Theta(h)$ such that
\begin{equation}\label{eq:frob_norm_rot_ball}
	\norm{V - U(\alpha)}_{F} \leq d(d-1)h.
\end{equation}
\end{proposition}
\begin{proof}
Proposition \ref{prop:recursiv_def_SOd} allows to parametrize $V$ by angles $\beta \in [0, 2 \pi)^{d-1} \times [0, \pi)^{(d-1)(d-2)/2}$ such that 
\[
V = V(\beta) = \prod_{r=1}^{d-1} \prod_{j=1}^{r} R(d - 1 - r + j, \beta^r_{j}) =: \prod_{k = 1}^{d(d-1)/2} R_k(\beta),
\] 
where in the last equality we introduced a short single indexed notation. Then, for any $\alpha \in \Theta(h)$, using the telescopic sum, we get
\begin{align}
	\norm{U(\alpha)- V(\beta)}_F 
	& = \Big\|\prod_{k=1}^{d(d-1)/2} R_k(\alpha) - \prod_{k=1}^{d(d-1)/2} R_k(\beta)\Big\|_F\\
	& = \Big\|\sum_{n=1}^{d(d-1)/2}  \Big[ \prod_{k=1}^{n-1} R_k(\alpha) \Big] \left( R_n(\alpha) - R_n(\beta) \right) \Big[ \prod_{k=n+1}^{d(d-1)/2} R_k(\beta) \Big]\Big\|_F\\
	&\leq 
	\sum_{n=1}^{d(d-1)/2}\Big\| \Big[ \prod_{k=1}^{n-1} R_k(\alpha) \Big] ( R_n(\alpha)-R_n(\beta)) \Big[ \prod_{k=n+1}^{d(d-1)/2} R_k(\alpha) \Big]\Big\|_F \\
	&= \sum_{n=1}^{d(d-1)/2} 
	\Big\|R_n(\alpha)-R_n(\beta)\Big\|_F .
\end{align}
The last equality follows from the fact that the Frobenius norm is invariant under left and right multiplication with orthogonal matrices. Next, recall that for every $n \in [d(d-1)/2]$ there exists a unique pair $r \in [d-1]$ and $j \in [r]$ such that
\begin{align}
	\norm{R_n(\alpha)-R_n(\beta)}_F^2 
	& =\norm{ R(d-1-r+j, \alpha^r_j) - R(d-1-r+j, \beta^r_j)}_{F}^2 \\
	& = 2(\cos(\alpha^r_j)-\cos(\beta^r_j))^2+2(\sin(\alpha^r_j)-\sin(\beta^r_j))^2
	\le 4 |\alpha^r_j - \beta^r_j|^2,
\end{align}
where used definition \eqref{eq: rotation def} of $R$ and that cosine and sine are Lipschitz-1 functions. By construction of $\Theta(h)$ there exists $\alpha \in \Theta(h)$, such that $|\alpha^r_j - \beta^r_j| < h$ for all $r,j$. Thus, we conclude that  
\[
\norm{U(\alpha)- V(\beta)}_F 
\le 2 \sum_{n=1}^{d(d-1)/2} |\alpha^r_j - \beta^r_j|
< 2 h d (d-1) / 2 = d(d-1)h.
\]
\end{proof}

In other words, $\Gamma$ is a $hd(d-1)$-net over $SO(d)$. Hence, for sufficiently small $h$, the minimizer of \eqref{eq:P_grid_search} will be in the neighborhood where local convergence of Riemannian gradient descent is guaranteed. 

\begin{corollary}\label{cor:convergence_global_min}
For $0 < h < 1$, let $U(h)$ denote the minimizer of \eqref{eq:P_grid_search} with $\Gamma(h)$ given by \eqref{eq: grid set}. Then, there exists $h>0$, such that $U(h)$ belongs to a level set defined in Theorem \ref{thm: manifold optimization local}. Consequently, the sequence generated by Riemannian gradient descent with step sizes as in Theorem \ref{thm: manifold optimization global} and initial guess $U^{(0)} = U(h)$ converges to a global minimum of $\ell_\varepsilon$.
\end{corollary}

\begin{proof}
By Theorem \ref{thm: manifold optimization local}, the all fixed points in the set $\mathcal L := \{U \in SO(d): \ell_\varepsilon(U) < \ell^* + q^*\}$ are global minimizers of $\ell_\varepsilon$ over $SO(d)$. It is open as a preimage $\ell_\varepsilon^{-1}( (-\infty, \ell^* + q^*) )$ of an open set via continuous function. Let $U_* \in \mathcal L$ be an arbitrary global minimizer of $\ell_\varepsilon$ over $SO(d)$. Since $\mathcal L$ is open, the open ball $\{ U \in SO(d) : \norm{U - U_*}_F < \delta \}$ is contained in $\mathcal L$ for some $\delta>0$. Then, by Proposition \ref{prop: grid search bound} with $h = \delta/ d (d-1)$, we can find $U_b \in \Gamma$ satisfying $\norm{U_b - U^*}_F < \delta$ and, hence, $U_b \in \mathcal L$. Consequently, the minimizer $U(h)$ of \eqref{eq:P_grid_search} admits 
\[
\ell_\varepsilon(U(h)) = \min_{U \in \Gamma} \ell_\varepsilon(U) \le \ell_\varepsilon(U_b) < \ell^* + q^*,
\]
so that $U(h) \in \mathcal L$. The rest of the claim follows from Theorem \ref{thm: manifold optimization local}.
\end{proof}

% \begin{remark}
% Since $U^{(0)}$ constructed via \eqref{eq:P_grid_search} is in $SO(d)$, the sequence $(U^{(k)})_{k\geq 0}$ generated by \eqref{eq: Riemannian gradient descent} remains in $SO(d)$ independently of the step size $\vartheta_k$ \OM{due to ...}.

% \OM{For the Landing method, the numerical examples studied in section~\ref{sect:numerics} show that there exists $k_0 \in \N$ such that for all $k\geq k_0, x_k \in SO(d)$. The proof of this assertion or a counterexample has to be done in a future work.}
% \end{remark}

We summarize the manifold optimization methods for in Algorithm \ref{alg: man opt}.

\begin{algorithm}
\begin{algorithmic}
	\State $\mathbf{Input}$ Diagonal blocks $B_n$ as in \eqref{eq:block}, $n \in [N]$, step sizes $\{ \nu_r \}_{r \ge 0}$, grid search density $h \in (0,1)$, smoothing $\varepsilon > 0$.
	\State Construct $\Gamma(h)$ as in \eqref{eq: grid set} and compute $U^{(0)}$ as the minimizer of \eqref{eq:P_grid_search}.
	\For{r = 0,1,\ldots}
	\State Compute $U^{(r+1)}$ via \eqref{eq: Riemannian gradient descent} or \eqref{eq: landing} with step sizes $\nu_r$.
	\State Check the stopping criteria.
	\EndFor
	\State \textbf{Return:} $\{ U^{(r)}\}_{r \ge 0}$. 
\end{algorithmic}
\caption{Manifold optimization for edge minimization}\label{alg: man opt}
\end{algorithm}

\subsection{Computational complexity}

The main drawback of computation complexity is the necessity to evaluate the function value for every $U \in \Gamma$. By construction \eqref{eq: grid set}, it leads to an exponential number of elements,
\begin{equation}
|\Theta(h)|=\left(\Bigl\lceil \frac{2\pi}{h} \Bigr\rceil\right)^{d-1}\left(\Bigl\lceil \frac{\pi}{h} \Bigr\rceil\right)^{\frac{(d-1)(d-2)}{2}}.
\end{equation}

Therefore, the grid search is only available for small dimension $d$ and requires efficient computation of matrices $U \in \Gamma$ and corresponding values $\ell_\varepsilon(U)$. 

Given that Algorithm \ref{alg: man opt} is applied to blocks $B^{U,k}$ in \eqref{eq:block}, the dimension of the problems is small as a result of splitting the graph into the finest connected components and optimizing the edges for each of them.

Furthermore, to avoid possible memory issues, the grid $\Theta(h)$ is split into blocks. First, the minimizer of $\ell(U)$ within the blocks is computed, and then $U$ in \eqref{eq:P_grid_search} is taken as the minimum among the blocks.  

As for the computation of $U$, the construction of rotation matrices \eqref{eq: rotation def} leads to a fast multiplication with other matrices. Let $A=(a_1 
\ldots \ a_d)  \in \R^{d \times d}$. Then, for arbitrary $r \in [d-1], \alpha \in [0, 2\pi)$, the product $AR(r, \alpha)$ has is given by
\[
\begin{pmatrix}
a_1 & \ldots & a_{r-1} & \cos(\alpha) a_r + \sin(\alpha) a_{r+1} & \cos(\alpha) a_{r+1} - \sin(\alpha) a_{r} & a_{r+1} & \ldots & a_{d}
\end{pmatrix},
\]
and requires only $4d$ operations. Using this, any $U \in \Gamma$ can be efficiently computed using $\mathcal O(d^3)$ operations.

The computational complexity of Algorithm \ref{alg: man opt} is summarized in Tabular \ref{tab:my_label}.

\begin{table}[h!]
\centering
\scalebox{0.96}{ \begin{tabular}{|l|l|c|}
		\hline
		\multirow{3}{*}{Grid search} & Number of grid points $\Theta(h)$& $\mathcal{O}\bigl((\lceil 2\pi/h \rceil)^{d}(\lceil \pi/h \rceil)^{d^2}\bigr)$\\
		\cline{2-3}
		& Matrix $U \in \Gamma$ & $\mathcal{O}(d^3)$ \\
		\cline{2-3}
		& Loss function $\ell(U)$ & $\mathcal{O}(Nd^3)$\\
		\cline{2-3}
		& Total & $\mathcal{O}\bigl(Nd^3(\lceil 2\pi/h \rceil)^{d}(\lceil \pi/h \rceil)^{d^2}\bigr)$ \\
		\hline
		\multirow{3}{*}{Manifold Optimiztion} & Riemannian/Euclidean gradient & $\mathcal{O}(N d^3)$\\
		\cline{2-3}
		& Riemannian gradient descent \eqref{eq: Riemannian gradient descent}& $\mathcal{O}(K N d^{3})$\\
		\cline{2-3}
		& Landing \eqref{eq: landing} & $\mathcal{O}(K N d^3)$\\
		\hline
\end{tabular}}
\caption{Number of operations for the grid search and for $K$ iterations of the manifold optimization method. Here $d$ is the dimension, $N$ the number of Hessians and $h$ the density of the grid.}
\label{tab:my_label}
\end{table}
%------------------------------------------

\section{Additional numerical results}\label{appendix:numerics}
A more detailed analysis for the results of Figure~\ref{fig:Truncation_random_hessians} are given for $\eta=10^{-9}$ (noise free Hessians), and $\eta=10^{-4}$ (noisy Hessians) in tables~\ref{tab:dim2}, \ref{tab:dim3}, \ref{tab:dim4}, \ref{tab:dim5}. They indicate that the sparsifying performance is better for the grid search initialization for small $h$ compared to random initialization for dimensions $d\in\{4,5\}$ for both clean and noisy data. Riemannian gradient descent performs better than the Landing method on clean data but the comparison on noisy data is inconclusive. Table~\ref{tab:runtime} shows that when the underlying dimension $d$ is small, there is no significant difference between the two methods regarding the time complexity. The only case where we can get improvement from Landing is runtime for large $d$ \cite{ablin2022fast}. In addition, as expected, the runtime effort for the calculation of the rotation matrices $U \in \Gamma(h)$ and the losses $\ell(U)$ increases with the dimension $d,$ the step size $h$ and the number of Hessians $N$. Furthermore, we see that the runtime of the grid search method for the step sizes considered in our numerics is still low and the grid-search initialization provides better results than the random initialisation method. To apply the grid search method more efficiently to more than 1 set of matrices, it is beneficial to calculate the rotation matrices only once, and reuse then for any set of matrices.
%which are basically independent of the test samples, and then determine the minimum loss $\ell(U), U \in \Gamma(U).$
\begin{table}[htbp]
\centering
\scalebox{0.82}{\begin{tabular}{|l|l|c|c|c|c|c|c|c|c|c|c|c|c|c|c|c|c|} 
		\cline{3-18}
		\multicolumn{2}{c|}{}&\multicolumn{8}{c|}{Clean data}&\multicolumn{8}{c|}{Noisy data $\sigma=10^{-3}$}\\
		\cline{3-18}
		\multicolumn{2}{c|}{}&\multicolumn{4}{c|}{Rgd}&\multicolumn{4}{c|}{Landing}&\multicolumn{4}{c|}{Rgd}&\multicolumn{4}{c|}{Landing}\\
		\cline{1-18}
		\multicolumn{2}{|c|}{$i$}&$0$&$1$&$2$&$\geq3$&$0$&$1$&$2$&$\geq3$&$0$&$1$&$2$&$\geq3$&$0$&$1$&$2$&$\geq3$\\
		\hline
		\multicolumn{2}{|c|}{Random initialization}&$100$&$0$&$0$&$0$& $100$&$0$&$0$&$0$ &$98$&$2$&$0$&$0$& $98$&$2$&$0$&$0$\\
		\hline
		\multirow{5}{*}{\makecell{Grid-Search \\ initialization}}&$h=1$&$100$&$0$&$0$&$0$& $100$&$0$&$0$&$0$& $98$&$2$&$0$&$0$& $98$&$2$&$0$&$0$\\ 
		& $h=0.5$& $100$&$0$&$0$&$0$ &$100$&$0$&$0$&$0$& $98$&$2$&$0$&$0$& $98$&$2$&$0$&$0$\\
		& $h=0.25$& $100$&$0$&$0$&$0$ & $100$&$0$&$0$&$0$& $98$&$2$&$0$&$0$ & $98$&$2$&$0$&$0$\\
		& $h=0.125$& $100$&$0$&$0$&$0$& $100$&$0$&$0$&$0$& $98$&$2$&$0$&$0$& $98$&$2$&$0$&$0$\\
		& $h=0.1$& $100$&$0$&$0$&$0$& $100$&$0$&$0$&$0$& $98$&$2$&$0$&$0$& $98$&$2$&$0$&$0$\\
		\hline
\end{tabular}}
\caption{\small{Number of set of matrices such that diff $\chi(U, \eta)=i$ (see \eqref{difference_sparsity}) and input dimension  $d=2$, $\eta=10^{-9}, 10^{-4}$ for clean resp. noisy data. Rgd: Riemannian gradient descent.}}
\label{tab:dim2}
\end{table}

\begin{table}[htbp]
\centering
\scalebox{0.82}{
	\begin{tabular}{|L{2cm}|L{1.65cm}|c|c|c|c|c|c|c|c|c|c|c|c|c|c|c|c|} 
		\cline{3-18}
		\multicolumn{2}{c|}{}&\multicolumn{8}{c|}{Clean data}&\multicolumn{8}{c|}{Noisy data $\sigma=10^{-3}$}\\
		\cline{3-18}
		\multicolumn{2}{c|}{}&\multicolumn{4}{c|}{Rgd}&\multicolumn{4}{c|}{Landing}&\multicolumn{4}{c|}{Rgd}&\multicolumn{4}{c|}{Landing}\\
		\cline{1-18}
		\multicolumn{2}{|c|}{$i$}&$0$&$1$&$2$&$\geq3$&$0$&$1$&$2$&$\geq3$&$0$&$1$&$2$&$\geq3$&$0$&$1$&$2$&$\geq3$\\
		\hline
		\multicolumn{2}{|l|}{Random initialization}&$98$&$0$&$0$&$2$& $87$&$0$&$5$&$8$& $98$&$1$&$1$&$0$& $95$&$2$&$3$&$0$\\
		\hline
		\multirow{5}{*}{\makecell{Grid-Search \\ initialization}}&$h=1$&$99$&$1$&$0$&$0$& $91$&$1$&$0$&$8$& $98$&$1$&$1$&$0$& $98$&$1$&$1$&$0$\\ 
		& $h=0.5$& $100$&$0$&$0$&$0$& $91$&$0$&$0$&$9$& $99$&$0$&$1$&$0$& $99$&$0$&$1$&$0$\\
		& $h=0.25$& $100$&$0$&$0$&$0$ &$91$&$0$&$0$&$9$& $99$&$0$&$1$&$0$& $99$&$0$&$1$&$0$\\
		& $h=0.125$& $100$&$0$&$0$&$0$& $92$&$0$&$0$&$8$& $99$&$0$&$1$&$0$& $99$&$0$&$1$&$0$\\
		& $h=0.1$& $100$&$0$&$0$&$0$& $92$&$0$&$0$&$8$& $100$&$0$&$0$&$0$& $100$&$0$&$0$&$0$\\
		\hline
\end{tabular}}
\caption{\small{Same as table~\ref{tab:dim2} for input dimension $d=3$.}}
\label{tab:dim3}
\end{table}

\begin{table}[htbp]
\centering
\scalebox{0.81}{\begin{tabular}{|L{2cm}|L{1.65cm}|c|c|c|c|c|c|c|c|c|c|c|c|c|c|c|c|} 
		\cline{3-18}
		\multicolumn{2}{c|}{}&\multicolumn{8}{c|}{Clean data}&\multicolumn{8}{c|}{Noisy data $\sigma=10^{-3}$}\\
		\cline{3-18}
		\multicolumn{2}{c|}{}&\multicolumn{4}{c|}{Rgd}&\multicolumn{4}{c|}{Landing}&\multicolumn{4}{c|}{Rgd}&\multicolumn{4}{c|}{Landing}\\
		\cline{1-18}
		\multicolumn{2}{|c|}{$i$}&$0$&$1$&$2$&$\geq3$&$0$&$1$&$2$&$\geq3$&$0$&$1$&$2$&$\geq3$&$0$&$1$&$2$&$\geq3$\\
		\hline
		\multicolumn{2}{|l|}{Random initialization}&$86$&$1$&$8$&$5$& $57$&$2$&$21$&$20$& $84$&$6$&$7$&$3$ & $88$&$1$&$11$&$0$\\
		\hline
		\multirow{5}{*}{\makecell{Grid-Search \\ initialization}}&$h=1$&$94$&$3$&$2$&$1$& $84$&$1$&$7$&$8$& $92$&$5$&$2$&$1$& $92$&$5$&$2$&$1$\\ 
		& $h=0.5$& $98$&$0$&$1$&$1$& $87$&$0$&$7$&$6$& $96$&$2$&$2$&$0$ & $96$&$2$&$2$&$0$\\
		& $h=0.25$& $100$&$0$&$0$&$0$& $89$&$0$&$6$&$5$& $99$&$1$&$0$&$0$ & $99$&$1$&$0$&$0$\\
		&$h=0.125$& $100$&$0$&$0$&$0$& $89$&$0$&$6$&$5$& $99$&$1$&$0$&$0$& $99$&$1$&$0$&$0$\\
		& $h=0.1$&$100$&$0$&$0$&$0$& $89$&$0$&$6$&$5$&$99$&$1$&$0$&$0$&$99$&$1$&$0$&$0$\\
		\hline
\end{tabular}}
\caption{\small{Same as table~\ref{tab:dim2} for input dimension $d=4$}}
\label{tab:dim4}
\end{table}

\begin{table}[htbp]
\centering
\scalebox{0.86}{\begin{tabular}{|L{2cm}|L{1.65cm}|c|c|c|c|c|c|c|c|c|c|c|c|c|c|c|c|} 
		\cline{3-18}
		\multicolumn{2}{l|}{}&\multicolumn{8}{c|}{Clean data}&\multicolumn{8}{c|}{Noisy data $\sigma=10^{-3}$}\\
		\cline{3-18}
		\multicolumn{2}{l|}{}&\multicolumn{4}{c|}{Rgd}&\multicolumn{4}{c|}{Landing}&\multicolumn{4}{l|}{Rgd}&\multicolumn{4}{c|}{Landing}\\
		\cline{1-18}
		\multicolumn{2}{|c|}{$i$}&$0$&$1$&$2$&$\geq3$&$0$&$1$&$2$&$\geq3$&$0$&$1$&$2$&$\geq3$&$0$&$1$&$2$&$\geq3$\\
		\hline
		\multicolumn{2}{|l|}{Random initialization}&$71$&$4$&$13$&$12$& $37$&$7$&$13$&$43$& $70$&$9$&$9$&$12$&$74$&$5$&$8$&$13$\\
		\hline
		\makecell{Grid-Search \\ initialization}&$h=1$& $93$&$2$&$4$&$1$&$85$&$2$&$8$&$5$&$93$&$2$&$4$&$1$&$93$&$2$&$4$&$1$\\ 
		\hline
\end{tabular}}
\caption{\small{Same as table~\ref{tab:dim2} for input dimension $d=5$}}
\label{tab:dim5}
\end{table}
\begin{table}[htbp]
\centering
\scalebox{0.7}{\begin{tabular}{|C{1.65cm}|c|c|c|r|r|r|r|r|r|r|r|r|r|} 
		\hline
		\multirow{3}{*}{\makecell{Input \\ dimension}}&\multicolumn{2}{c|}{Random Init.}&\multicolumn{10}{c|}{Grid-search Init.}\\
		\cline{2-13}
		&\multirow{2}{*}{Rgd}&\multirow{2}{*}{La}&\multicolumn{5}{c|}{Rotation $U \in \Gamma(h)$}&\multicolumn{5}{c|}{Losses $\ell(U), U \in \Gamma(U)$}\\
		\cline{4-13}
		&&&$1$&$0.5$&$0.25$&$0.125$&$0.1$&$1$&$0.5$&$0.25$&$0.125$&$0.1$\\
		\hline
		\multirow{2}{*}{$d=2$}&\multirow{2}{*}{$29.36$}&\multirow{2}{*}{$33.36$}&\multirow{2}{*}{$0.0010$}&\multirow{2}{*}{$0.0009$}&\multirow{2}{*}{$0.0008$}&\multirow{2}{*}{$0.0008$}&\multirow{2}{*}{$0.0008$}&$0.0003$&$0.0003$&$0.0003$&$0.0003$&$0.0003$\\
		\cline{9-13}
		&&&&&&&&$0.0009$&$0.0009$&$0.0009$&$0.0009$&$0.0009$ \\
		\hline
		\multirow{2}{*}{$d=3$}&\multirow{2}{*}{$29.39$}&\multirow{2}{*}{$32.18$}&\multirow{2}{*}{$0.0211$}&\multirow{2}{*}{$0.0211$}&\multirow{2}{*}{$0.0211$}&\multirow{2}{*}{$0.0213$}&\multirow{2}{*}{$0.0215$}&$0.0032$&$0.0032$&$0.0032$&$0.0035$&$0.0041$\\
		\cline{9-13}
		&&&&&&&&$0.0194$&$0.0193$&$0.0193$&$0.0210$&$0.0244$ \\
		\hline
		\multirow{2}{*}{$d=4$}&\multirow{2}{*}{$29.66$}&\multirow{2}{*}{$32.50$}&\multirow{2}{*}{$0.6105$}&\multirow{2}{*}{$0.7062$}&\multirow{2}{*}{$2.0015$}&\multirow{2}{*}{$11.3731$}&\multirow{2}{*}{$168.4473$}&$0.0343$&$0.0761$&$0.1732$&$5.9366$&$12.4121$\\
		\cline{9-13}
		&&&&&&&&$0.3434$&$0.7612$&$1.7323$&$59.3655$&$124.1206$ \\
		\hline
		\multirow{2}{*}{$d=5$}&\multirow{2}{*}{$29.77$}&\multirow{2}{*}{$32.55$}&\multirow{2}{*}{$23.5289$}&\multirow{2}{*}{$-$}&\multirow{2}{*}{$-$}&\multirow{2}{*}{$-$}&\multirow{2}{*}{$-$}&$1.1344$&$-$&$-$&$-$&$-$\\
		\cline{9-13}
		&&&&&&&&$17.0158$&$-$&$-$&$-$&$-$ \\
		\hline
\end{tabular}}
\caption{\small{Mean runtime in seconds$(s)$ for the $5$-time random initialisation method with $K=5\cdot 10^3$ iterations and the grid search initialisation (step size $h=1, 0.5, 0.25, 0.125, 0.1$) procedure over $10$ random chosen examples out of the $100$ randomly generated examples for dimension $d=2, 3, 4, 5$. For each dimension $d=2,3,4,5$ the first row in the subtable losses $\ell(U), U \in \Gamma(U)$ denotes the minimal runtime for $N=1$ Hessians and the second row the maximal runtime for $N=d(d+1)/2$ Hessians. Rgd: Riemannian gradient descent. La: Landing method.}}
\label{tab:runtime}
\end{table}
\end{document}